\documentclass[12pt]{amsart}

\usepackage{mathrsfs}

\usepackage{amsfonts, amssymb,amsmath, amsthm, amsxtra, latexsym, amscd}
\usepackage[latin1]{inputenc} 

\usepackage[english]{babel}

\usepackage{amsmath}

\usepackage{amsthm,amssymb,latexsym}

\usepackage{t1enc}

\usepackage{epic}

\usepackage{eepic}

\usepackage{xypic}

\usepackage{amsfonts, amssymb,amsmath, amsthm, amsxtra, latexsym, amscd}

\usepackage{graphics}

\title[Branching of some holomorphic representations of SO(2,n)]{Branching 
of some holomorphic representations of SO(2,n)}
\author{Henrik Seppänen}
\address{Henrik Sepp\"{a}nen, Fachbereich Mathematik, AG AGF
Technische Universit\"at Darmstadt
Schlo{\ss}gartenstra{\ss}e 7
64289 Darmstadt}
\email{seppaenen@mathematik.tu-darmstadt.de}
\keywords{ Bounded symmetric domains, Lie groups, Lie algebras, unitary 
representations, spherical functions,
hypergeometric functions, intertwining operator}
\subjclass{32M15, 22E46, 22E43, 43A90, 32A36}

\newtheorem{prop}{Proposition}

\newtheorem{lemma}[prop]{Lemma}

\newtheorem{cor}[prop]{Corollary}

\newtheorem{thm}[prop]{Theorem}

\theoremstyle{definition}

\newtheorem{defin}[prop]{Definition}

\theoremstyle{remark} \newtheorem*{remark}{Remark}

\begin{document}

\maketitle

\begin{abstract}
In this paper we consider the analytic continuation of the weighted 
Bergman spaces on the Lie ball
$$\mathscr{D}=SO(2,n)/S(O(2) \times O(n))$$ and the corresponding 
holomorphic unitary (projective) representations of $SO(2,n)$ on
these spaces. These representations are known to be irreducible. Our aim 
is to decompose them under the subgroup $SO(1,n)$ which
acts as the isometry group of a totally real submanifold $\mathscr{X}$ of 
$\mathscr{D}$.
We give a proof of a general decomposition theorem for certain unitary 
representations of semisimple Lie groups.
In the particular case we are concerned with, we find an explicit formula 
for the Plancherel measure of the decomposition
as the orthogonalising measure for certain hypergeometric polynomials.
Moreover, we construct an explicit generalised Fourier transform that 
plays the role of the intertwining operator for the
decomposition. We prove an inversion formula and a Plancherel formula for 
this transform. Finally we construct explicit
realisations of the discrete part appearing in the decomposition and also 
for the minimal representation in this family.
\end{abstract}

\section*{Introduction}
One of the main problems in the representation theory of Lie groups and harmonic 
analysis on Lie groups is to decompose some interesting representations of 
a Lie group
$G$ under a subgroup $H \subset G$. This decomposition is also 
called the \emph{branching rule}. Among other things, this has led to 
the discovery of new interesting representations.  An exposition of the general theory for compact connected
Lie groups, including the classical results for $U(n)$ and $SO(n)$ (by Weyl and Murnaghan respectively),
can be found in \cite{knapp-beyond}.

Since the work by R. Howe \cite{howedual} and M. Kashiwara and M. Vergne (cf \cite{k-w}), it 
has turned
out to be fruitful to study the branching of singular and minimal 
holomorphic representations of a
Lie group acting on a function space of holomorphic functions on a bounded 
symmetric domain. In \cite{jak-v-restr}, Jakobsen and Vergne study the restriction of 
the tensor product of two holomorphic representations to the diagonal subgroup.

In this paper we will study the branching of the analytic continuation of 
the scalar holomorphic discrete series of $SO(2,n)$ under the subgroup 
$H=SO_0(1,n)$. The subgroup $H$ here is
realised as the isometry group of a totally real submanifold of the Lie 
ball $SO(2,n)/S(O(2) \times O(n))$.
The branching for a general Lie group $G$ of Hermitian type under a 
symmetric subgroup $H$ has been studied recently by Neretin (\cite{ner}, \cite{nerbeta}),
Zhang (\cite{zhtams},\cite{zhSB},\cite{zhtp})
and by van Dijk and Pevzner \cite{dijkpevz}. In \cite{kob-orsted-minbranch}, Kobayashi and {\O}rsted
studied the branching for some minimal representations.  
The branching
rule for regular parameter and for some minimal representations is now 
well understood.
However, the problem of finding the branching rule for non-discrete, non-regular parameter
is a difficult one, and there is still no complete theory for the general case.

We find the branching rule for arbitrary scalar parameter $\nu$ in the 
Wallach set of $SO(2,n)$. It turns out that for small parameters $\nu$ there appears a discrete part in the decomposition. 
We discover here an intertwining operator realising the corresponding representation.
It should be mentioned that for large parameter (in this case $\nu > n-1$) the corresponding branching problem
has been solved by Zhang in \cite{zhtams} for arbitrary bounded symmetric domains. 
 
The paper is organised as follows. In Section 1 we describe the geometry of 
the Lie ball. In Section 2 we
recall some facts about general bounded symmetric domains and Jordan triple 
systems. In Section 3 we establish some facts
about the real part of the Lie ball. In Section 4 we consider a family of 
function spaces and corresponding unitary
representations. Section 5 is devoted to branching theorems and to finding 
the Plancherel measure. In Sections 6 and
7 we find realisations of the representations corresponding to the 
discrete part in the decomposition and to the minimal
point in the Wallach set respectively.\\\\
\noindent {\bf Acknowledgement.} I would like to thank Prof.\ Genkai Zhang 
for his guidance and support during the
preparation of this paper and for all the discussions we have had on 
related topics. It is also my pleasure to thank
Prof.\ Andreas Juhl for having read an earlier version of this manuscript 
and for having provided valuable comments.
Finally, I would like to thank the anonymous referee for valuable comments, especially
for having brought to my attention several interesting papers on branching that I was not aware of. 

\section{The Lie ball as a symmetric space $SO_0(2,n)/SO(2) \times SO(n)$}
In this paper we study representations on function spaces on the domain
\begin{eqnarray}
\mathscr{D}=\{z \in \mathbb{C}^n | 1-2\langle z,z \rangle + |zz^t|^2 > 0, 
|z|<1\}.
\end{eqnarray}
We will only be concerned with the case $n>2$. (If $n=1$ it is the unit 
disk, $U$, and if $n=2$, $\mathscr{D} \cong U \times U$). In this section 
we describe $\mathscr{D}$ as the quotient of $SO_0(2,n)$ by $SO(2) \times 
SO(n))$ by studying a holomorphically equivalent model
on which we have a natural group action induced by the linear action on a submanifold of a
Grassmanian manifold.
Consider $\mathbb{R}^{n+2}\cong \mathbb{R}^{2} \oplus \mathbb{R}^{n}$ 
equipped with the non-degenerate bilinear form
\begin{eqnarray*}
(x|y):=x_{1}y_{1}+x_{2}y_{2}-x_{3}y_{3}- \ldots -x_{n+2}y_{n+2},
\end{eqnarray*}
where the coordinates are with respect to the standard basis $e_1, \cdots, 
e_{n+2}$. Let $SO(2,n)$ be the group of
all linear transformations on $\mathbb{R}^{n+2}$ that preserve this form 
and have determinant 1, i.e.,
\begin{eqnarray*}
SO(2,n)=\{g \in GL(2+n,\mathbb{R}) | (gx|gy)=(x|y), x,y \in 
\mathbb{R}^{2+n}, \det g=1\}
\end{eqnarray*}
Let $\mathcal{G}_{(2,n)}^{+}$ denote the set of all two-dimensional 
subspaces of $\mathbb{R}^{2} \oplus \mathbb{R}^{n}$
on which $(\cdot |\cdot )$ is positive definite.
Clearly $\mathbb{R}^2 \oplus \{0\}$ is one of these subspaces. It will be 
the reference point in $\mathcal{G}_{(2,n)}^{+}$ and
we will denote it by $V_0$.
The group $SO(2,n)$ acts naturally on this set and the action is  
transitive. In fact, the connected component of the identity, $SO_0(2,n)$ acts transitively. We will let $G$
denote this group.

We denote by $K$ the stabilizer subroup of $V_0$, 
i.e.,
\begin{eqnarray}
K=\{g \in G |g(V_0)=V_0\}.
\end{eqnarray}
Any element $g \in G$ can be identified with a $(2+n)\times(2+n)$- 
matrix of the form
\begin{eqnarray}
\left( \begin{array}{cc}
A & B\\
C & D \label{matrix group}
\end{array} \right),
\end{eqnarray} 
where $A$ is a $2 \times 2$-matrix.
With this identification, $K$ clearly corresponds to the matrices
\begin{eqnarray*}
\left( \begin{array}{cc}
A & 0\\
0 & D
\end{array} \right),
\end{eqnarray*}
where $A$ and $D$ are orthogonal $2 \times 2$- and $n \times n$-matrices with deerminant one 
respectively, i.e.,
$K \cong SO(2) \times SO(n)$.
The space $\mathcal{G}_{(2,n)}^{+}$ can be realised as the unit ball in 
$M_{n2}(\mathbb{R})$ with the operator norm. Indeed,
let $V \in \mathcal{G}_{(2,n)}^{+}$. If $v=v_1+v_2 \in V$, then $v_1=0$ 
implies that $v_2=0$, i.e., the projection $v \mapsto v_1$ is an injective mapping. This means that 
there is a real $n \times 2$ matrix $Z$
with $Z^tZ < I_2$, such that
\begin{eqnarray}
V=\{(v \oplus Z v) | v \in \mathbb{R}^2\}. \label{matris}
\end{eqnarray}
Conversely, if $Z \in M_{n2}(\mathbb{R})$ satisfies $Z^tZ < I_2$, then 
\eqref{matris}
defines an element in $\mathcal{G}_{(2,n)}^{+}$. 

Using \eqref{matrix group} to identify $g$ with a matrix and letting $V$ 
correspond to the matrix $Z$, then clearly
\begin{eqnarray*}
gV&=&\{(Av + BZv \oplus Cv + DZ v) | v \in \mathbb{R}^2)\}\\
&=&\{v \oplus (C + DZ)(A+BZ)^{-1} v) | v \in \mathbb{R}^2)\}.
\end{eqnarray*}
In other words, we have a $G$-action on the set
\begin{eqnarray*}
M=\{Z \in M_{n2}(\mathbb{R}) | Z^tZ < I_2\}
\end{eqnarray*}
given by

\begin{eqnarray*}
Z \mapsto (C+DZ)(A+BZ)^{-1}.
\end{eqnarray*}
This exhibits $M$ as a symmetric space. $$M \cong G/K.$$

Moreover, we identify the matrix $Z=(X Y)$ with the vector $X+iY$ in 
$\mathbb{C}^n$ in order to obtain an almost complex structure on $M$.
With respect to this almost complex structure, the action of $G$ is in fact holomorphic.
Moreover we have the following 
result by Hua (see \cite{hua2}).

\begin{thm}
The mapping
\begin{eqnarray*}
\mathcal{H}: z \mapsto Z=2
\left(\left( \begin{array}{cc}
zz^t+1 & i(zz^t-1)\\
\overline{z}\overline{z}^t+1 & -i( \overline{z}\overline{z}^t-1)
\end{array} \right)^{-1}
\left(\begin{array}{c}
z\\
\overline{z}
\end{array}\right)
\right)^t,
\end{eqnarray*}
where $zz^t=z_1^2+ \cdots +z_n^2$, is a holomorphic diffeomorphism of the bounded domain
\begin{eqnarray*}
\mathscr{D}=\{z \in \mathbb{C}^n | 1-2\langle z,z \rangle + |zz^t|^2 > 0, 
|z|<1\}
\end{eqnarray*}
onto $M$.
\end{thm}

We will call this mapping the \emph{Hua transform}. It allows us to describe $\mathscr{D}$ as a symmetric space
\begin{eqnarray*}
\mathscr{D} \cong M \cong G/K.
\end{eqnarray*}


\begin{section}{Bounded symmetric domains and Jordan pairs}
In this section we review briefly some general theory on bounded symmetric 
domains and Jordan pairs. All proofs
are omitted. For a more detailed account we refer to Loos (\cite{loos}) and to Faraut-Koranyi (\cite{FK}).

Let $\mathcal{D}$ be a bounded open domain in $\mathbb{C}^n$ and 
$\mathcal{H}^2(\mathcal{D})$ be the Hilbert space of all
square integrable holomorphic functions on $\mathcal{D}$,
\begin{eqnarray*}
\mathcal{H}^2(\mathcal{D})=\{f, f \,\text{holomorphic on}\,\ \mathcal{D}\,| 
\int_{\mathcal{D}}|f(z)|^2dm(z)<\infty\},
\end{eqnarray*}
where $m$ is the $2n$-dimensional Lebesgue measure.
It is a closed subspace of $L^2(\mathcal{D})$. For
every $w \in \mathcal{D}$, the evaluation functional $ f \mapsto f(w)$ is 
continuous, hence $\mathcal{H}^2(\mathcal{D})$
has a reproducing kernel $K(z,w)$, holomorphic in $z$ and antiholomorphic 
in $w$ such that

\begin{eqnarray*}
f(w)=\int_{\mathcal{D}}f(z)\overline{K(z,w)}dm(z).
\end{eqnarray*}
$K(z,w)$ is called the Bergman kernel. It has the transformation
property
\begin{eqnarray}
K(\varphi(z),\varphi(w))=J_{\varphi}(z)^{-1}K(z,w)\overline{J_{\varphi}(w)}^{-1},
\end{eqnarray}
for any biholomorphic mapping $\varphi$ on $\mathcal{D}$ with complex 
Jacobian $J_{\varphi}(z)=\det d\varphi(z)$.
Hereafter biholomorphic mappings will be referred to as automorphisms.
The formula
\begin{eqnarray}
h_z(u,v)=\partial_u\partial_{\overline{v}}\log K(z,z)
\end{eqnarray}
defines a Hermitian metric, called the Bergman metric. It is invariant 
under automorphisms and
its real part is a Riemannian metric on $\mathcal{D}$.

A bounded domain $\mathcal{D}$ is called \emph{symmetric} if, for each $z 
\in \mathcal{D}$ there is an involutive
automorphism $s_z$ with $z$ as an isolated fixed point.
Since the group of automorphisms, $Aut(\mathcal{D})$ preserves the Bergman 
metric, $s_z$ coincides with
the local geodesic symmetry around $z$. Hence $\mathcal{D}$ is a Hermitian 
symmetric space.

A domain $\mathcal{D}$ is called \emph{circled} (with respect to 0) if $0 
\in \mathcal{D}$ and $e^{it}z \in
\mathcal{D}$ for every $z \in \mathcal{D}$ and real $t$.

Every bounded symmetric domain is holomorphically isomorphic with a 
bounded symmetric and circled domain. It is
unique up to linear isomorphisms.

From now on $\mathcal{D}$ denotes a circled bounded symmetric domain. $G$ 
is the identity component of
$Aut(\mathcal{D})$, $K$ is the isotropy group of $0$ in $G$. The Lie 
algebra $\mathfrak{g}$ will
be considered as a Lie algebra of holomorphic vector fields on 
$\mathcal{D}$, i.e., vector fields $X$ on $\mathcal{D}$
such that $Xf$ is holomorphic if $f$ is.
The symmetry $s ,z \mapsto -z$  around the origin induces an invoulution 
on $G$ by $g \mapsto sgs^{-1}$ and, by
differentiating, an involution $Ad(s)$ of $\mathfrak{g}$. We have the 
Cartan decomposition
\begin{eqnarray*}
\mathfrak{g}=\mathfrak{k} \oplus \mathfrak{p}
\end{eqnarray*}
into the $\pm 1$-eigenspaces.

For every $v \in \mathbb{C}^n$, let $\xi_v$ be the unique vector field in 
$\mathfrak{p}$
that takes the value $v$ at the origin. Then
\begin{eqnarray}
\xi_v(z)=v-Q(z)\overline{v} \label{vectorfield}
\end{eqnarray}
where  $Q(z): \overline{V} \rightarrow V$ is a complex linear mapping and 
$Q: V \rightarrow $ Hom$(\overline{V},V)$ is a
homogeneous quadratic polynomial.
Hence $Q(x,z)=Q(x+z)-Q(x)-Q(z): \overline{V} \rightarrow V$ is bilinear 
and symmetric in $x$ and $z$. For $x,y,z \in V$, we define
\begin{eqnarray}
\{x\overline{y}z\}=D(x,\overline{y})z=Q(x,z)\overline{y}
\end{eqnarray}
Thus $\{x\overline{y}z\}$ is complex bilinear and symmetric in $x$ and $z$ 
and complex antilinear in $y$, and $D(x,\overline{y})$ is
the endomorphism $z \mapsto \{x\overline{y}z\}$ of $V$.

The pair 
$(V,\{\,\,\})$ is called a \emph{Jordan triple system}.
This Jordan triple system is positive
in the sense that
if $v \in V, v \neq 0$ and $Q(v)\overline{v}=\lambda v$ for some $\lambda 
\in \mathbb{C}$, then $\lambda$ is positive.
We introduce the endomorphisms
\begin{eqnarray}
B(x,y)=I-D(x,\overline{y})+Q(x)\overline{Q}(\overline{y})
\end{eqnarray}
of $V$ for $x,y \in V$, where $
\overline{Q}(\overline{y})x=
\overline{Q({y})\bar x}$.
We summarise some results in the following proposition.
\begin{prop}
a) The Lie algebra $\mathfrak{g}$ satisfies the relations
\begin{eqnarray}
\left[\xi_u,\xi_v\right]&=&D(u,\overline{v})-D(v,\overline{u})\\
\left[l,\xi_{u}\right]&=&\xi_{lu}
\end{eqnarray}
for $u,v \in V$ and $l \in \mathfrak{k}$\\
b) The Bergman kernel $k(x,y)$ of $\mathcal{D}$ is
\begin{eqnarray}
m(\mathcal{D})^{-1}\det B(x,y)^{-1}
\end{eqnarray}
c) The Bergman metric at $0$ is
\begin{eqnarray}
h_0(u,v)=\mbox{tr} D(u,\overline{v}),
\end{eqnarray}
and at an arbitrary point $z\in \mathcal{D}$
\begin{eqnarray}
h_z(u,v)=h_0(B(z,z)^{-1}u,v)
\end{eqnarray}
d) The triple product $\{\,\,\}$ is given by
\begin{eqnarray}
h_0(\{u\overline{v}w\},y)=\partial_u\partial_{\overline{v}}\partial_x\partial_{\overline{y}} 
\log K(z,z)|_{z=0}
\end{eqnarray}
\end{prop}
We define odd powers of an element $x \in V$ by
\begin{eqnarray*}
x^1=x, \,\,x^3=Q(x)\overline{x}, \cdots, x^{2n+1}=Q(x)\overline{x^{2n-1}}.
\end{eqnarray*}

An element $x \in V$ is said to be \emph{tripotent} if $x^3=x$, i.e., if 
$\{x\overline{x}x\}=2x$.
Two tripotents $c$ and $e$ are called orthogonal if $D(c,\overline{e})=0$. 
In this case $D(c, \overline{c})$ and $D(e,\overline{e})$ commute
and $e+c$ is a tripotent. 

Every $x\in V$ can be written uniquely
\begin{eqnarray*}
x=\lambda_1c_1+ \cdots +\lambda_n c_n,
\end{eqnarray*}
where the $c_i$ are pairwise orthogonal nonzero tripotents which are real 
linear combinations of odd powers of $x$, and the $\lambda_i$ satisfy
\begin{eqnarray*}
0 < \lambda_1 < \cdots <\lambda_n.
\end{eqnarray*}
This expression for $x$ is called its \emph{spectral decomposition} and 
the $\lambda_i$ the eigenvalues of $x$.
Moreover, the domain $\mathcal{D}$
can be realised as the unit ball in $V$ with the spectral norm
\begin{eqnarray*}
\|x\|=\max |\lambda_i|,
\end{eqnarray*}
where the $\lambda_i$ are the eigenvalues of $x$, i.e.,
\begin{eqnarray*}
\mathcal{D}=\{x \in V | \|x\| < 1\}.
\end{eqnarray*}
Let $f(t)$ be an odd complex valued function of the real variable $t$, 
defined for $|t|<\rho$. For every $x \in V$ with $|x|<\rho$ we define
$f(x) \in V$ by
\begin{eqnarray}
f(x)=f(\lambda_1)c_1+\cdots +f(\lambda_n)c_n,
\end{eqnarray}
where $x=\lambda_1c_1+\cdots +\lambda_nc_n$ is the spectral resolution of 
$x$.
This functional calculus is used in expressing the action on $\mathcal{D}$ 
of the elements $\exp \xi_v$ in $G$:
\begin{eqnarray}
\exp \xi_v(z)=u+B(u,u)^{1/2}B(z,-u)^{-1}(z+Q(z)\overline{u})
\end{eqnarray}
and
\begin{eqnarray}
d(\exp \xi_v)(z)=B(u,u)^{1/2}B(z,-u)^{-1},
\end{eqnarray}
where $u=\tanh v$, for $v \in \mathbb{C}^n$ and $z\in \mathcal{D}$.
\end{section}

\begin{section}{The real part of the Lie ball}
We consider the non-degenerate quadratic form
\begin{eqnarray}
q(z)=z_1^2+\cdots +z_n^2
\end{eqnarray}
on $V=\mathbb{C}^n$. In the following we will often denote $q(z,w)$ by 
$(z,w)$. Defining $Q(x)y=q(x,y)x-q(x)y$, where $q(x,y)=q(x+y)-q(x)-q(y)$,
we get a Jordan triple system.
The  Lie ball $\mathscr{D}=\{z \in \mathbb{C}^n | 1-2\langle z,z \rangle + 
|zz^t|^2 > 0, |z|<1\}$ is the open unit
ball in this Jordan triple system.
An easy computation shows the following identity.
\begin{eqnarray*}
D(x,\overline{y})z=2(\sum_{k=1}^nx_k\overline{y_k})z+2(\sum_{k=1}^nz_k\overline{y_k})x-2(\sum_{k=1}^nx_kz_k)\overline{y}\\
\end{eqnarray*}
Recalling that 
$B(x,y)=I-D(x,\overline{y})+Q(x)\overline{Q}(\overline{y})$.
The Bergman kernel of $\mathscr{D}$ is
\begin{eqnarray}
K(z,w)=(1-2\langle z, w \rangle + (zz^t)\overline{(ww^t)})^{-n}.
\end{eqnarray}
We will hereafter denote it by $h(z,w)^{-n}$.
Consider the real form $\mathbb R^n$ in $\mathbb C^n$. Observe that 
$$\mathscr{X}:=\mathscr{D} \bigcap \mathbb{R}^n$$ is the unit ball of $\mathbb{R}^n$.
On $\mathscr{X}$ we have a simple expression for the Bergman metric:
\begin{eqnarray}
B(x,x)=(1-|x|^2)^{-2}I, x \in \mathscr{X}. \label{bergman}
\end{eqnarray}
The submanifold $\mathscr{X}$ is a totally real form of 
$\mathscr{D}$ in the sense that
\begin{eqnarray*}
T_x(\mathscr{X})+iT_x(\mathscr{X})=T_x(\mathscr{D}),\,\,
T_x(\mathscr{X}) \bigcap iT_x(\mathscr{X})=\{0\}
\end{eqnarray*}
This implies that every holomorphic function on $\mathscr{D}$ that 
vanishes on $\mathscr{X}$ is identically zero. 
We define the subgroup $H$ as the identity component of
\begin{eqnarray*}
\{h\in G |h(x) \in \mathscr{X} \,\text{if} \,x\in \mathscr{X}\}
\end{eqnarray*}
We will denote $H \bigcap K$ by $L$.

Using the fact that the real form $\mathbb{R}^n$ is a sub-triple system of $\mathbb{C}^n$, one can show that
$\mathscr{X}$ is a totally geodesic submanifold of $\mathscr{D}$ (cf Loos \cite{loos}). 
Hence we can describe $\mathscr{X}$
as a symmetric space $$\mathscr{X} \cong H/L.$$

We now study the image of $\mathscr{X}$ in the $M_{n2}(\mathbb{R})$- model 
of the Lie ball. For computational convenience, we now work with the 
transposes of these matrices.
The defining equation of the Hua-transform can be written as
\begin{eqnarray}
\frac{1}{2}
\left( \begin{array}{cc}
zz^t+1 & i(zz^t-1)\\
\overline{z}\overline{z}^t+1 & -i( \overline{z}\overline{z}^t-1)
\end{array} \right)
Z=\left( \begin{array}{c}
          z\\
   \overline{z}
    \end{array} \right)
\end{eqnarray}
In the coordinates $(z_1, \ldots, z_n)$ of $z$, this identity takes the form
\begin{eqnarray}\label{hua1}
z_k=\frac{1}{2}(\;(zz^t+1)x_k + i(zz^t-1)y_k).
\end{eqnarray}
This gives
\begin{eqnarray}
4zz^t=(zz^t)^2(X+iY)(X+iY)^t+2(XX^t+YX^t)zz^t\\
+(X-iY)(X-iY)^t,
\end{eqnarray}
which is a quadratic equation in $zz^t$ with unique solution
\begin{eqnarray}\label{hua2}
&&zz^t\\
&&=\frac{2-(XX^t+YY^t)-2\sqrt{(1-XX^t)(1-YY^t)-(YX^t)^2}}{(X+iY)(X+iY)^t}.\nonumber
\end{eqnarray}
From \eqref{hua1} we see that if $z$ is real, then $y_k=0$ for all $k$. On 
the other hand, if $Y=0$, then \eqref{hua2} shows that $zz^t$ is
real and therefore $z$ is real by \eqref{hua1}.
Hence the image of the real part $\mathscr{X} \subset \mathscr{D}$ under 
the Hua-transform is the set
\begin{eqnarray}
\mathcal{H}(\mathscr{X})=\{Z=\left(X \,0\right) | X \in 
M_{n1}(\mathbb{R}), |X|<1\},
\end{eqnarray}
since for an element $Z=\left(X \,0\right)$, the condition that $Z^tZ < I_2$ is 
clearly equivalent with $|X| < 1$.

Recall that the real $n$-dimensional unit ball can be described as a 
symmetric space $SO_0(1,n)/SO(n)$ by a procedure analogous to the one in 
the first section. One first considers all lines in $\mathbb{R}^{1+n}$ on 
which the quadratic form $x_1^2-x_2^2- \cdots -x_{n+1}^2$ is positive 
definite
and identifies these lines with all real $n \times 1$-matrices with norm 
less than one.
If we write elements $g \in SO(1,n)$ as matrices of the form
\begin{eqnarray}
g=\left(\begin{array}{cccc}
a & - & b & -\\
| & & &\\
c & & D &\\
|& & & \\ \end{array}\right),
\end{eqnarray}
the action is given by
\begin{eqnarray}
X \mapsto (c+DX)(a+bX)^{-1}.
\end{eqnarray}
The group $SO(1,n)$ can be embedded into $SO(2,n)$. Indeed, the equality
\begin{multline*}
\left(\begin{array}{ccccc}
a& 0 & - & b & -\\
0 & 1 & - & 0 & -\\
| & | &   &   &  \\
c & 0 &   & D &  \\
| & | &   &   &  \\ \end{array}
\right)
\left(\begin{array}{ccccc}
a'  & 0 & - & b'  & -\\
0  & 1& - & 0 & -\\
|  & | &   &    &  \\
c' & 0 &   & D' &  \\
|  & | &   &    &  \\ \end{array}
\right)\\\\
= \left(\begin{array}{ccccc}
aa'+bc'  & 0 & - & ab'+bD' & -\\
0  & 1 & - & 0 & -\\
| & | &   &   &  \\
ca'+Dc' & 0 &   & cb'+DD' & \\
| & | & & &  \\ \end{array}
\right)
\end{multline*}
shows that we can define an injective homomorphism $\theta: SO(1,n) 
\rightarrow SO(2,n)$ by
\begin{eqnarray*}
\theta:
\left(\begin{array}{cccc}
a & - & b & -\\
| & & &\\
c & & D &\\
|& & & \\ \end{array}\right)
\mapsto
\left(\begin{array}{ccccc}
a& 0 & - & b & -\\
0 & 1 & - & 0 & -\\
| & | &   &   &  \\
c & 0 &   & D &  \\
| & | &   &   &  \\ \end{array}.
\right)
\end{eqnarray*}
This subgroup acts on $\mathcal{H}(\mathscr{X})$ as
\begin{eqnarray*}
\left(X \, 0 \right) \mapsto \left((c+DX)(a+bX)^{-1} \, \,0\right)
\end{eqnarray*}
and the action is transitive.
Suppose now that $h \in SO(2,n)$ preserves $H(\mathscr{X})$. Let $p=h(0)$. 
We can choose a $g \in SO_0(1,n)$ such that $g(0)=p$ (here we identify $g$ 
with
$\theta(g)$). Then $g^{-1}h(0)=0$ and hence we can write it in block form 
as
\begin{eqnarray*}
g^{-1}h=\left( \begin{array}{cc}
I_2 & 0\\
0 & D\\
\end{array}\right),
\end{eqnarray*}
with $D \in SO(n)$. This is an element in $\theta(SO(1,n))$ and hence $h 
\in \theta(SO(1,n))$.
We have now proved the following theorem.

\begin{thm}
The Hua transform $\mathcal{H}:\mathscr{D} \rightarrow M$ maps the real 
part $\mathscr{X}$ diffeomorphically onto
\begin{eqnarray}
\mathcal{H}(\mathscr{X})=\{Z=\left(X \,0\right) | X \in 
M_{n1}(\mathbb{R}), |X|<1\}
\end{eqnarray}
by $x \mapsto \frac{2x}{1+|x|^2}$.
Moreover, the induced group homomorphism $h \mapsto \mathcal{H}\, h\, 
\mathcal{H}^{-1}$ is an isomorphism between
the groups $H$ and $SO_0(1,n)$
\end{thm}

\begin{remark}
The model $\mathcal{H}(\mathscr{X})$ of $SO_0(1,n)/SO(n)$ is the real part 
of the complex $n$-dimensional unit ball $SU(1,n)/SU(n)$ with fractional- 
linear group action. It is therefore equipped with a Riemannian metric 
given by the restriction of the Bergman metric of the complex unit ball.
If $x \in  \mathcal{H}(\mathscr{X}), x \neq 0$, we decompose 
$\mathbb{R}^n=\mathbb{R}x \oplus (\mathbb{R}x)^{\perp}$. We let
$v=v_x+v_{x^{\perp}}$ be the corresponding decomposition of a tangent 
vector $v$ at $x$.
In this model, the Riemannian metric at $x$ is (cf \cite{rudinft})
\begin{eqnarray*}
g_x(v,v)=\frac{|v_x|^2}{(1-|x|^2)^2}+\frac{|v_{x^{\perp}}|^2}{(1-|x|^2)}.
\end{eqnarray*}
We recall from equation \eqref{bergman} that if $x \in X$, then the 
Riemannian metric at $x$ is
\begin{eqnarray*}
h_x(v,v)=\frac{1}{2n}\frac{|v|^2}{(1-|x|^2)^2}.
\end{eqnarray*}
The Hua transform thus induces an isometry (up to a constant) of the real 
$n$-dimensional unit ball equipped with two different Riemannian 
structures.
\end{remark}

\subsection{Iwasawa decomposition of $\mathfrak{h}$}
The Cartan decomposition $\mathfrak{g}=\mathfrak{k} \oplus \mathfrak{p}$ 
induces a decomposition
$\mathfrak{h}=\mathfrak{l} \oplus \mathfrak{q}$.
We let $$\mathfrak{a}=\mathbb{R}\xi_e$$ be the one-dimensional subspace of $\mathfrak{q}$, where
$e=e_1$ denotes the first standard basis vector and the corresponding vector field $\xi_e$ is defined in 
\eqref{vectorfield}.
\begin{prop}
The Lie algebra $\mathfrak{h}$ has rank one, and the roots with respect to 
the abelian subalgebra $\mathfrak{a}$
of $\mathfrak{q}$ are
$\{\alpha,-\alpha\}$, where $\alpha(\xi_{e})=2$. The corresponding
posive root space is
\begin{eqnarray*}
\mathfrak{q}_{\alpha}=\{\xi_v+\frac{1}{2}(D(e,v)-D(v,e)) | v \in 
\mathbb{R}e_2 \oplus \cdots \oplus \mathbb{R}e_n\}
\end{eqnarray*}
\end{prop}

\begin{proof}
This is known in a general context, but we give here an elementary 
proof.

Take $u$ and $v$ in $\mathbb{R}^n$ and assume that $[\xi_u,\xi_v]=0$. 
Then, for any $x \in \mathbb{R}^n$ we have
\begin{eqnarray*}
D(u,v)x=D(v,u)x.
\end{eqnarray*}
A simple calculation shows that this amounts to
\begin{eqnarray*}
(u,x)v=(v,x)u,
\end{eqnarray*}
which can only hold for all real $x$ if $u=v$.

Thus $\mathfrak{a}$ is a maximal abelian subalgebra in $\mathfrak{q}$.
The vector $e$ is a maximal tripotent in the Jordan triple system 
corresponding to $\mathscr{D}$.
Suppose that $[\xi_e,\xi_v+l]=\alpha(\xi_e)(\xi_v+l)$. Identifying the 
$q$- and $l$-components yields
\begin{eqnarray}
D(e,v)-D(v,e)=\alpha(\xi_e)l \label{l}\\
-\xi_{le}=\alpha(\xi_e)\xi_v \label{q}
\end{eqnarray}
From \eqref{q} it follows that $le=-\alpha(\xi_e)v$ and, thus, applying 
both sides of $\eqref{l}$ to $e$ gives
\begin{eqnarray*}
D(e,v)e-D(v,e)e=-\alpha(\xi_e)^2v,
\end{eqnarray*}
i.e.,
\begin{eqnarray*}
D(e,e)v-D(e,v)e=\alpha(\xi_e)^2v,
\end{eqnarray*}
An easy computation gives
\begin{eqnarray*}
4v-4(e,v)e=\alpha(\xi_e)^2v.
\end{eqnarray*}
Hence $e$ is orthogonal to $v$ and $\alpha(\xi_e)^2=4$. The rest follows immediately.
\end{proof}
We shall fix the positive root $\alpha$. Elements in
$\mathfrak{a}_{\mathbb{C}}^*$ are of the form $\lambda \alpha$ and
will hereafter be identified with the complex numbers $\lambda$. In
particular, the half sum of the positive roots (with
multiplicities), $\rho$, will be identified with the number
$(n-1)/2$.

\subsection{The Cayley transform}
The Cayley transform is a biholomorphic mapping from a bounded symmetric 
domain onto a \emph{Siegel domain}. We describe it for the domain 
$\mathscr{D}$
and use it to express the spherical functions on $\mathscr{X}$ in terms of 
the spherical functions on the unbounded domain.
We fix the maximal tripotent $e$. Then $\mathbb{C}^n$ equipped with the 
bilininear mapping
\begin{eqnarray}
(z,w) \mapsto z \circ w=\frac{1}{2}\{z e w\}
\end{eqnarray}
is a complex Jordan algebra. Observe that since $e$ is a tripotent, it is 
a unity for this multiplication.
The Cayley transform is the mapping $c: \mathbb{C}^n \rightarrow 
\mathbb{C}^n$ defined by
\begin{eqnarray}
c(z)=(e+z)\circ(e-z)^{-1},
\end{eqnarray}
where $(e-z)^{-1}$ denotes the inverse of $(e-z)$ with respect to the 
Jordan product.
\begin{prop}
The Cayley transform is given by the formula
\begin{eqnarray}
c(z)=\frac{1-zz^t}{1-2z_1+(zz^t)^2}e+\frac{2z'}{1-2z_1+(zz^t)^2},\label{cayley}
\end{eqnarray}
for $z=(z_1,z')=z_1e+z'\in \mathscr{D}.$ Moreover, it maps $\mathscr{X}$ onto the 
halfspace
\begin{eqnarray*}
\{(x_1, \ldots, x_n\}\in \mathbb{R}^n|x_1>0\}.
\end{eqnarray*}
\end{prop}

\begin{proof}
We first find the inverse for an element $x$. Suppose therefore that 
$e=\frac{1}{2}\{xez\}=\frac{1}{2}D(x,e)z$,   i.e.,
\begin{eqnarray*}
e=(x,e)z+(z,e)x-(x,z)e=x_1z+z_1x-(x,z)e
\end{eqnarray*}
Identifying coordinates gives
\begin{eqnarray*}
1=2x_1z_1-(x,z)\label{1koord}\\
0=x_1z'+z_1x' \label{koord2}
\end{eqnarray*}
These equations have the solution
\begin{eqnarray*}
z_1&=&x_1/(x,x)\\
z'&=&-x'/(x,x).
\end{eqnarray*}
If we apply this to the expression $(e-z)^{-1}$ in the definition of $c$, 
we get
\begin{eqnarray*}
(e-z)^{-1}=\frac{1-z_1}{(1-z_1)^2+(z',z')}e+\frac{z'}{(1-z_1)^2+(z',z')}.
\end{eqnarray*}
Now the formula \eqref{cayley} follows by an easy computation. Moreover, 
we observe that the inverse transform is given by
\begin{eqnarray*}
w \mapsto (w-e)\circ(w+e)^{-1}=-c(-w).
\end{eqnarray*}
Hence both $c$ and $c^{-1}$ preserve $\mathbb{R}^n$ and therefore
\begin{eqnarray*}
c(\mathscr{X})=c(\mathscr{D}) \bigcap \mathbb{R}^n.
\end{eqnarray*}
We now determine $c(\mathscr{X})$.

From (\cite{loos}) we know that (since $e$ is a 
maximal tripotent)
\begin{eqnarray}
c(\mathscr{D})=\{u+iv |u \in A^+, v \in A\},
\end{eqnarray}
where $A$ is the real Jordan algebra $$\{z \in V|Q(e)\overline{z}=z\}$$ 
and $A^+$ is the positive cone $\{z \circ z | z \in A\}$ in $A$.
By a simple computation we see that $$A=\mathbb{R}e\oplus \mathbb{R}ie_2 
\oplus \cdots \oplus \mathbb{R}ie_n.$$
Since we have the identities $$z+Q(e)\overline{z}=2u,$$ 
$$z-Q(e)\overline{z}=2iv$$ and 
$$Q(e)\overline{z}=2\overline{z_1}-\overline{z},$$
we get expressions for $u$ and 
$v$:$$2u=(z_1+\overline{z_1},z_2-\overline{z_2}, 
\ldots,z_n-\overline{z_n})$$
$$2iv=(z_1-\overline{z_1},z_2+\overline{z_2},\ldots, z_n+\overline{z_n})$$
The condition that $x=u+iv$ be in the image of $\mathscr{X}$ thus implies that $$u=(x_1,0, \ldots, 
0),$$ $$iv=(0,x_2, \ldots, x_n).$$
Moreover we require that
\begin{eqnarray*}
u= w \circ w=2w_1w-(w,w)e,
\end{eqnarray*}
for some $$w=c_1e+c_2ie_2+ \cdots +c_n ie_n.$$
This yields $$(x_1, \ldots,0)=(c_1^2+ \cdots +c_n^2,ic_1c_2, \ldots, 
ic_1c_n).$$
Hence
\begin{eqnarray*}
c_1^2=x_1, c_2=\cdots=c_n=0,
\end{eqnarray*}
and thus $$u+iv=(c_1^2,x_2, \ldots, x_n).$$ This proves the claim.
\end{proof}
Recall the expression for the spherical functions on a symmetric space of 
noncompact type (cf \cite{helg2} Thm 4.3)
\begin{eqnarray*}
\varphi_{\lambda}(h)=\int_Le^{(i\lambda+\rho)A(lh)}dl,
\end{eqnarray*}
where $A(lh)$ is the (logarithm) of the $A$ part of $lh$ in the Iwasawa 
decomposition $H=NAL$. The integrand in this formula is called the
\emph{Harish-Chandra e-function}. For the above Siegel domain it has the 
form $e_{\lambda}(w)=(w_1)^{i\lambda+\rho}$ (cf \cite{untup}). Hence we 
have the
following corollary.
\begin{cor}\label{cayleycor}
The spherical function $\varphi_{\lambda}$ on $\mathscr{X}=H/L$ is
\begin{eqnarray}
\varphi_{\lambda}(x)=\int_{S^{n-1}}\left(\frac{1-|x|^2}{1-2(x,\zeta)+xx^t}\right)^{i\lambda+\rho}d\sigma(\zeta).
\end{eqnarray}
where $\sigma$ is the $O(n)$-invariant probability measure on $S^{n-1}$.
\end{cor}
\end{section}

\begin{section}{A family of unitary representations of $G$}

\subsection{The function spaces $\mathscr{H}_{\nu}$}
The Bergman space $\mathcal{H}^2(\mathscr{D})$ has the reproducing kernel 
$h(z,w)^{-n}$.
This means in particular that the function $h(z,w)^{-n}$ is positive 
definite in the sense that

\begin{eqnarray*}
\sum_{i,j=1}^m \alpha_i\overline{\alpha_j}h(z_i,z_j)^{-n} \geq 0,
\end{eqnarray*}
for all $z_1, \ldots, z_n \in \mathscr{D}$ and $\alpha_1, \ldots, \alpha_n 
\in \mathbb{C}.$
It has been proved by Wallach 
(\cite{wall})
and Rossi-Vergne (\cite{Rossi-V}) that $h(z,w)^{-\nu}$ is positive 
definite precisely when $\nu$
in the set
\begin{eqnarray*}
\{ 0,(n-2)/2\} \bigcup \left((n-2)/2, \infty \right)
\end{eqnarray*}
This set will also be referred as the \emph{Wallach set} (cf 
\cite{FKsymm}).
For $\nu$ in the Wallach set above, $h(z,w)^{-\nu}$ is the reproducing 
kernel of a Hilbert space of holomorphic
functions on $\mathscr{D}$. We will call this space $\mathscr{H}_{\nu}$ 
and the reproducing kernel $K_{\nu}(z,w)$.
The mapping $g \mapsto \pi_{\nu}(g)$, where

\begin{eqnarray*}
\pi_{\nu}(g)f(z)=J_{g^{-1}}(z)^{\frac{\nu}{n}}f(g^{-1}z)
\end{eqnarray*}
defines a unitary projective representation of $G$ on $\mathscr{H}_{\nu}$.
Indeed, comparison with the Bergman kernel shows that $h(z,w)^{-\nu}$ 
transforms under automorphisms according to the rule
\begin{eqnarray}\label{kernel}
h(gz,gw)^{-\nu}=J_g(z)^{-\frac{\nu}{n}}h(z,w)^{-\nu}\overline{J_g(w)}^{-\frac{\nu}{n}}.
\end{eqnarray}
Recall that for functions $f_1$ and $f_2$ of the form
\begin{eqnarray*}
f_1(z)=\sum_{k=1}^l\alpha_kK_{\nu}(z,w_k), 
\,\,f_2(z)=\sum_{k=1}^m\beta_kK_{\nu}(z,w'_k),
\end{eqnarray*}
the inner product is defined as
\begin{eqnarray}
\langle f_1,f_2 
\rangle_{\nu}=\sum_{i,j}\alpha_i\overline{\beta_j}K_{\nu}(w_i,w'_j) 
\label{inner}
\end{eqnarray}
Equation \eqref{kernel} implies that
\begin{eqnarray}
K_{\nu}(g^{-1}z,w)=J_{g^{-1}}(z)^{-\frac{\nu}{n}}K_{\nu}(z,gw)\overline{J_{g^{-1}}(w)}^{-\frac{\nu}{n}}.
\end{eqnarray}
Hence we have the following two equalities
\begin{eqnarray*}
\pi_{\nu}(g)f_1(z)=\sum_{k=1}^l\alpha_k\overline{J_{g^{-1}}(w_k)}^{-\frac{\nu}{n}}K_{\nu}(z,gw_k)\\
\pi_{\nu}(g)f_2(z)=\sum_{k=1}^m\beta_k\overline{J_{g^{-1}}(w'_k)}^{-\frac{\nu}{n}}K_{\nu}(z,gw'_k).
\end{eqnarray*}
The unitarity
\begin{eqnarray*}
\langle \pi_{\nu}(g)f_1,\pi_{\nu}(g)f_2 \rangle_{\nu}=\langle f_1,f_2 
\rangle_{\nu}
\end{eqnarray*}
now follows by an application of the transformation
rule \eqref{kernel} in the definition \eqref{inner}.
Since functions of the form above are dense in $\mathscr{H}_{\nu}$, it 
follows that each $\pi_{\nu}(g)$ is a unitary
operator and it is easy to see that $g \mapsto \pi_{\nu}(g)$ is a 
projective homomorphism of groups.
In fact, $\pi_{\nu}$ is an irreducible projective representation, cf
\cite{FK}.

\subsection{Fock-Fischer spaces}
It can be shown that for $\nu > (n-2)/2$ all holomorphic polynomials are 
in $\mathscr{H}_{\nu}$ and that polynomials of different homogeneous
degree are orthogonal. In this context, the spaces $\mathscr{H}_{\nu}$ are 
closely linked with the \emph{Fock-Fischer space}, $\mathscr{F}$,
which we will now describe.
The basis vector $e_1$ is a maximal tripotent which is decomposed into 
minimal tripotents as
$e_1=\frac{1}{2}(1,i,0, \ldots, 0)+\frac{1}{2}(1,-i,0,\ldots,0).$ (We omit 
the easy computations.) In order to expand the reproducing
kernel $K_{\nu}$ into a power series consistent with the treatment in 
\cite{FK}, we need to introduce a new norm on $\mathbb{C}^n$ so
that the minimal tripotents have norm $1$, i.e., the Euclidean norm 
multiplied with $\sqrt{2}$. Then
\begin{eqnarray*}
\{f_1, \ldots, f_n\}:=
\left\{\frac{1}{\sqrt{2}}e_1, \ldots,\frac{1}{\sqrt{2}}e_n\right\}
\end{eqnarray*}
is an orthonormal basis with respect to this new norm. We write points 
$z\in \mathscr{D}$ as $z=w_1f_1+ \cdots +w_nf_n$.
For polynomials $p(w)=\sum_{\alpha}a_{\alpha}w^{\alpha}$, we define
\begin{eqnarray*}
p^*(w)=\sum_{\alpha}\overline{a_{\alpha}}w^{\alpha}.
\end{eqnarray*}
The Fock-Fischer inner product is now defined as

\begin{eqnarray*}
\langle p,q \rangle_{\mathscr{F}}=p(\partial)(q^*)|_{w=0},
\end{eqnarray*}
where $p(\partial)$ is the differential operator
$\sum_{\alpha}a_{\alpha}\frac{\partial^{\alpha}}{\partial
w^{\alpha}}$, for $p$ as above. The Fock-Fischer space, $\mathscr{F}$, is the
completion of the space of polynomials. It
is easy to see that polynomials of different homogeneous degree are
orthogonal in $\mathscr{F}$. Moreover, the representation of $SO(n)$
on $\mathcal{P}^m$, the polynomials of homogeneous degree $m$, can be
decomposed into irreducible subspaces as

\begin{eqnarray}
\mathcal{P}^m=\bigoplus_{m-2k \geq 0} E_{m-2k}\otimes \mathbb{C}(ww^t)^k,
\end{eqnarray}
where $E_{i}$ are the spherical harmonic polynomials of degree $i$
(cf \cite{stein-w}). This is a special case of the general Hua-Schmid decomposition 
(cf \cite{FK}).The following relation holds between the
Fock-Fischer norm and the $\mathscr{H}_{\nu}$-norm on the space
$E_{m-2k}\otimes \mathbb{C}(ww^t)^k$ (cf \cite{FK}).

\begin{eqnarray}
\|p\|_{\nu}^2=\frac{\| p 
\|_{\mathscr{F}}^2}{(\nu)_{m-k}(\nu-\frac{n-2}{2})_k},\label{normer}
\end{eqnarray}
for $p \in E_{m-2k}\otimes \mathbb{C}(ww^t)^k$.
We have the following decomposition of $\mathscr{H}_{\nu}$ under $K$:
\begin{prop}(Faraut-Kor\`anyi, \cite{FK}) \label{FK}
a) If $\nu > \frac{n-2}{2}$, then
\begin{eqnarray}
\mathscr{H}_{\nu}|_K=\bigoplus\sum_{m-2k \geq 0} E_{m-2k} \otimes \mathbb{C}(zz^t)^k,
\end{eqnarray}
where $E_{m-2k}$ is the space of spherical harmonic polynomials of degree $m-2k$.
Moreover, we have the following expansion of the kernel function:
\begin{eqnarray}
h(z,w)^{-\nu}=\sum_{m-2k \geq 
0}(\nu)_{m-k}\left(\nu-\frac{n-2}{2}\right)_k K_{(m-k,k)}(z,w),
\end{eqnarray}
where $K_{(m-k,k)}$ is the reproducing kernel for the subspace 
$E_{m-2k}\otimes \mathbb{C}(zz^t)^k$ with the Fock-Fischer norm.
The series converges in norm
and uniformly on compact sets of $\mathscr{D} \times \mathscr{D}$.\\
b) If  $\nu = \frac{n-2}{2}$, then
\begin{eqnarray}
\mathscr{H}_{\nu}|_K=\bigoplus \sum_mE_m
\end{eqnarray}
\end{prop}
We will later need the norm of $(zz^t)^k$ in $\mathscr{H}_{\nu}$.
\begin{prop}
\begin{eqnarray}
\|(zz^t)^k\|_{\nu}^2=\frac{k!\left(\frac{n}{2}\right)_k}{(\nu)_k\left(\nu-\frac{n-2}{2}\right)_k}
\end{eqnarray}
\end{prop}

\begin{proof}
A straightforward computation shows that
\begin{eqnarray*}
(\frac{\partial^2}{\partial z_1^2}+ \cdots + \frac{\partial^2}{\partial 
z_n^2})(z_1^2 + \cdots + z_n^2)^k
=(2^2 k(k-1)+n 2k)(z_1^2+ \cdots + z_n^2)^{k-1}
\end{eqnarray*}
Proceeding inductively, we obtain
\begin{eqnarray*}
\left(\frac{\partial^2}{\partial z_1^2}+ \cdots + 
\frac{\partial^2}{\partial z_n^2}\right)^k(z_1^2 + \cdots + z_n^2)^k
&=& \prod_{j=1}^k 2j(2(j-1)+n)\\
&=& 4^kk!\left(\frac{n}{2}\right)_k
\end{eqnarray*}
The Fock-Fischer norm is computed in the $w$-coordinates 
$w_i=\sqrt{2}z_i$, so
\begin{eqnarray*}
(zz^t)^k=2^{-k}(ww^t)^k
\end{eqnarray*}
 and
\begin{eqnarray*}
\left(\frac{\partial^2}{\partial z_1^2}+ \cdots + 
\frac{\partial^2}{\partial z_n^2}\right)^k
=2^{-k}\left(\frac{\partial^2}{\partial w_1^2}+ \cdots + 
\frac{\partial^2}{\partial w_n^2}\right)^k.
\end{eqnarray*}
Hence
\begin{eqnarray*}
\|(zz^t)^k\|_{\mathscr{F}}^2=k!\left(\frac{n}{2}\right)_k
\end{eqnarray*}
and an application of Prop. \ref{FK} gives the result.
\end{proof}
\end{section}

\begin{section}{Branching of $\pi_{\nu}$ under the subgroup $H$}

\subsection{A decomposition theorem}
Recall the irreducible (projective) representations $\pi_{\nu}$ from
the previous section. Our main objective is to decompose these into
irreducible representations under the subgroup $H$. The fact that 
$\mathscr{X}$
is a totally real form is reflected in the
restrictions of the representations $\pi_{\nu}$ to $H$.
\begin{prop}
The constant function $1$ is in $\mathscr{H}_{\nu}$ and is an 
$L$-invariant cyclic vector for the representation
$\pi_{\nu}: H \rightarrow \mathscr{U}(\mathscr{H}_{\nu}).$
\end{prop}

\begin{proof}
First note that
\begin{eqnarray*}
K_{\nu}(z,h0)&=&J_h(h^{-1}z)^{-\nu/n}K_{\nu}(h^{-1}z,0)\overline{J_{h}(0)^{-\nu/n}}\\
&=&\overline{J_h(0)^{-\nu/n}}J_{h^{-1}}(z)^{\nu/n}K_{\nu}(h^{-1}z,0)\\
&=&\overline{J_h(0)^{-\nu/n}}\pi_{\nu}(h)1(z)
\end{eqnarray*}
Suppose now that the function $f \in \mathscr{H}_{\nu}$ is orthogonal to 
the linear span of the elements $\pi_{\nu}(h)1, h \in H$. By the above 
identity we have

\begin{eqnarray*}
f(h0)&=&\langle f,K_{\nu}(\cdot,h0)\rangle_{\nu}\\
&=&0.
\end{eqnarray*}
Since $H$ acts transitively on $\mathscr{X}$, $f$ is zero on 
$\mathscr{X}$. Hence it is identically zero.
\end{proof}
We want decompose the representation of $H$ into a \emph{direct integral} 
of irreducible representations.
For the definition of a direct integral over a measurable field of Hilbert 
spaces we refer to Naimark (\cite{naim}).
The following general decomposition theorem is stated in several 
references (e.g. \cite{ner}), but the author has not been able to find a proof of it in the 
literature. A proof for abelian groups can be found in \cite{naim}.
The proof we present below is based on the Gelfand-Naimark representation 
theory for $C^*$-algebras.

\begin{thm} \label{abstr}
Let $\pi$ be a unitary representation of the semisimple Lie group $H$ on 
a Hilbert space, $\mathscr{H}$.
Suppose further that $L$ is a maximal compact subgroup and that the 
representation has a cyclic $L$-invariant vector.
Then $\pi$ can be decomposed as a multiplicity-free direct integral of 
irreducible representations,
\begin{eqnarray}
\pi \cong \int_{\Lambda}\pi_{\lambda\,}d\mu(\lambda),
\end{eqnarray}
where $\Lambda$ is a subset of the set of positive definite spherical 
functions on $H$ and for $\lambda \in \Lambda$,$\,\pi_{\lambda}$
is the corresponding unitary spherical representation.
\end{thm}

\begin{proof}
We consider the Banach space $L^1(H)$. This is a Banach $*$-algebra with multiplication defined as 
the convolution
$$(f*g)(x)=\int_Hf(y)g(y^{-1}x)dy$$ and involution defined by
$$f^{*}(x)=\overline{f(x^{-1})}.$$
Recall that the representation $\pi$ extends to a representation of the 
Banach algebra $L^1(H)$ by
\begin{eqnarray*}
f \mapsto \int_Hf(x)\pi(x)dx.
\end{eqnarray*}
We will also denote this mapping of $L^1(H)$ into 
$\mathscr{B}(\mathscr{H})$ (the set of bounded linear operators on $\mathscr{H}$) by $\pi$.
This representation will also be cyclic as the following argument shows. 
Denote by $\xi$ the  
$L$-invariant cyclic unit vector for $H$.
Vectors of the form
\begin{eqnarray*}
\pi(f_\epsilon)(\pi(h_1)\xi+ \cdots +\pi(h_n)\xi),
\end{eqnarray*}
where $\{f_\epsilon\}$ is an approximate identity on $H$, will then be dense in $\mathscr{H}$. Moreover the identity
\begin{eqnarray}
\pi(f)(\pi(h_1)\xi+ \cdots +\pi(h_n)\xi)=\pi((R_{h_1^{-1}}+ \cdots 
+R_{h_n^{-1}})f)\xi \nonumber,
\end{eqnarray}
holds for $f \in L^1(H)$ and $h_1, \ldots, h_n \in H$. (Here $R_h f$ denotes the right-translation of the argument of $f$; 
$f \mapsto f(\cdot \,h)$. We similarly define $L_hf$.) 
Hence vectors of the form $\pi(f)\xi$, where $f \in L^1(H)$, form a dense subset in $\mathscr{H}$.

The function $\Phi$ defined as
\begin{eqnarray}
\Phi: \pi(f) \mapsto \langle \pi(f)\xi,\xi \rangle
\end{eqnarray}
extends to a state on the
$C^*$-algebra $\mathscr{C}$ generated by $\pi(L^1(H))$ and the identity operator. 
It is a well-known fact from 
the theory of
$C^*$-algebras that the norm-decreasing positive functionals form a convex 
and weak*-compact set (cf \cite{mur}).
For a $C^*$-algebra with identity, the extreme points of this set are the pure states. 
Therefore, $\Phi$ can be expressed as
\begin{eqnarray}
\Phi=\int_X \varphi_x d\mu, \label{states}
\end{eqnarray}
where $X$ is the set of pure states and $\mu$ is a regular Borel measure on 
$X$ (cf \cite{rudfa}, Thm. 3.28).
We recall the Gelfand-Naimark-Segal construction of a cyclic 
representation of a $C^*$-algebra
associated with a given state (cf \cite{mur}). In this duality, the irreducible 
representations correspond to
the pure states. So each $\varphi_x$ in \eqref{states} parametrises an 
irreducible representation
of $\pi(L^1(H))$ on some Hilbert space $H_x$ with a $\pi(L^1(H))$-cyclic unit vector $\xi_x$.

Herafter we will, by an abuse of notation, write $\Phi(f)$ for $\Phi(\pi(f))$ and correspondingly for the functionals $\varphi_x$.

We define a unitary operator $T: \mathscr{H} \rightarrow 
\int_X H_xd\mu$ that
intertwines the actions of $\mathscr{C}$ by
\begin{eqnarray}
T:\pi(f)\xi \mapsto \{\pi_x(f)\xi_x\}, f \in L^1(H).
\end{eqnarray}
To see that this is well-defined, suppose that $\pi(f)\xi=0$. Then we have
\begin{eqnarray}
\langle \pi(f)\xi,\pi(f)\xi \rangle=\langle \pi(f^**f)\xi,\xi\rangle=0
\end{eqnarray}
i.e.,
\begin{eqnarray}
\Phi(f^**f)=0
\end{eqnarray}
By \eqref{states} we have
\begin{eqnarray}
\Phi(f^**f)=\int_H \langle \pi_x(f^**f)\xi_x,\xi_x\rangle_xd\mu=0. 
\label{isom}
\end{eqnarray}
Therefore $\pi_x(f)\xi_x=0$ for almost every $x$ and hence $T$ is well 
defined on a
dense set of vectors. Note that \eqref{isom} also shows that $T$ is 
isometric on this
set and it therefore extends to an isometry of $\mathscr{H}$ into 
$\int_X H_xd\mu$.

Consider now the subalgebra, $L^1(H)^\#$, consisting of all 
$L^1$-functions that are left- and right
$L$-invariant, i.e.,
\begin{eqnarray*}
L_lf=R_lf=f,
\end{eqnarray*}
for all $l$ in $L$. This is a commutative Banach *-algebra (cf \cite{helg2}, Ch. IV).
We know that $\varphi_x \circ \pi: L^1(H)^\# 
\rightarrow \mathbb{C}$ is
a homomorphism of algebras and is therefore of the form (cf \cite{helg2},
Ch. IV)
\begin{eqnarray}
\varphi_x(f)=\int_Hf(h)\phi_x(h)dh, f \in L^1(H)^{\#},
\end{eqnarray}
where $\phi_x$ is a bounded spherical function. 
In fact, this formula holds for all $L^1$-functions on $H$, as the following argument shows.

Since $\xi$ is $L$-invariant, the identity $$\pi(f)\xi=\pi(R_lf)\xi$$ holds for all $L^1$-functions $f$
and $l \in L$. Applying $T$ to both sides of this equality (and using the fact that both $L^1(H)$ and $L$ are separable),
we see that 
\begin{eqnarray}
\pi_x(f)\xi_x=\pi_x(R_lf)\xi_x \label{rightinvstates}
\end{eqnarray}  
holds for all $f \in L^1(H)$ and $l \in L$ outside some set of measure zero with respect to $\mu$. 
We now choose an approximation of the identity $\{\eta_{\epsilon}\}$ on $H$, and by replacing it with
$\{\int_L \eta_{\epsilon}(l \cdot l^{-1})dl \}$ if necessary, we may assume that it is invariant under the conjugate
action of $L$.

Consider now $\Phi_{\epsilon}$ defined by 
$$\varphi_{\epsilon}(f)=\langle \pi(f)\pi(\eta_{\epsilon})\xi, \pi(\eta_{\epsilon})\xi \rangle.$$ 
We define the functionals $\varphi_{x,\epsilon}$ analogously for all $x \in X$.
Clearly $\Phi_{\epsilon}(f) \rightarrow \Phi(f)$ as $\epsilon \rightarrow 0$ and therefore 
$$\lim_{\epsilon \rightarrow 0}\varphi_{x,\epsilon}(f)=\varphi_x(f)$$ holds for all $L^1$-functions $f$ outside some 
set of measure zero with respect to $\mu$. (Again we use the separability of $L^1(H)$.)
Using the $L$-conjugacy invariance of $\eta_{\epsilon}$ and \eqref{rightinvstates}, a simple calculation shows that
$$\varphi_{x,\epsilon}(f)=\varphi_{x,\epsilon}(f^{\#}),$$
where $$f(h)=\int_L\int_Lf(l_1hl_2)dl_1dl_2,$$ and by letting $\epsilon$ tend to zero we get
$$\varphi_x(f)=\varphi_x(f^{\#})$$ for almost every $x$. Hence
$$\varphi_x(f)=\int_Hf(h)\phi_x(h)dh,$$
for $f \in L^1(H)$.

Since $\varphi_x$ also preserves the involution $*$, it is
a positive linear functional, i.e.,
\begin{eqnarray}
\int_H f(h)\phi_x(h)dh \geq 0, \label{pos}
\end{eqnarray}
for every $f \in L^1(H)$, such that $f=g*g^*$, for some $g \in 
L^1(H)$.

\begin{lemma}
Suppose that $\varphi$ is a bounded spherical function such that\\
$\int_H f(h) \varphi(h)dh \geq 0$ for all $f \in L^1(H)$ of the form 
$f=g*g^*$
for some $g \in L^1(H)$.
Then $\varphi$ is positive definite.
\end{lemma}

\begin{proof}
For any $f=g*g^*$ as in the statement, we have
\begin{eqnarray*}
\int_H f(h) \varphi(h)dh=\int_H\int_H g(y)\overline{g(h^{-1}y)}dy
\varphi(h)dh\\
=\int_H g(y)\int_H \overline{g(h^{-1}y)}\varphi(h)dhdy\\
=\int_H\int_H g(y)\overline{g(z)}\varphi(yz^{-1})dzdy
\end{eqnarray*}
Pick any complex numbers $c_1, \ldots, c_n$ and elements $x_1, \ldots, 
x_n$
in $H$ and fix $\epsilon > 0$. We can choose a compact set $K \subset H$,
containing all $x_i$ in its interior, and a neighbourhood $U$ of the
identity such that
\begin{eqnarray}
|\varphi(xy^{-1})-\varphi(x'y'^{-1})|< \epsilon
\end{eqnarray}
for all $(x,y)$ and $(x',y')$ in $K \times K$
such that
$(xx'^{-1},yy'^{-1}) \in U \times U.$
We now choose disjoint neighbourhoods $E_i$ of $x_i$ such that
$E_i \subset K$ and $x_i^{-1}E_i \subset U$ for all $i$.
Next, we choose an $L^1$-function $g$ in such a way that its support lies
in $\bigcup E_i$ and $g$ has the constant value $c_i/|E_i|$ on $E_i$,
where

\begin{eqnarray*}
|E_i|=\int_{E_i}dh.
\end{eqnarray*}
Now we have
\begin{eqnarray}
\int_H \int_H g(y)\overline{g(z)}\varphi(yz^{-1})dzdy=
\sum_{i,j}\int_{E_i}\int_{E_j} g(y)\overline{g(z)}\varphi(yz^{-1})dzdy 
\nonumber\\
=\sum_{i,j}g(y_i)\overline{g(y_j)}\varphi(y_iy_j^{-1})|E_i||E_j|,
\label{sum}
\end{eqnarray}
for some $y_i$ in $E_i$.
The choice of $g$ implies that the sum in \eqref{sum} equals
\begin{eqnarray*}
\sum_{i,j}c_i\overline{c_j}\varphi(y_iy_j^{-1}).
\end{eqnarray*}
This yields
\begin{eqnarray*}
|\int_H\int_H 
g(y)\overline{g(z)}\varphi(yz^{-1})dzdy-\sum_{i,j}c_i\overline{c_j}\varphi(x_ix_j^{-1})|\\
<\sum_{i,j}|c_i||c_j||\varphi(y_iy_j^{-1})-\varphi(x_ix_j^{-1})|<n^2 
\sup_i|c_i|\epsilon.
\end{eqnarray*}
This shows that
\begin{eqnarray*}
\sum_{i,j}c_i\overline{c_j}\varphi(x_ix_j^{-1}) \geq 0,
\end{eqnarray*}
and hence $\varphi$ is positive definite.
\end{proof}
Since every positive definite spherical function defines an irreducible, 
unitary,
spherical representation of $H$, it also gives rise to a representation 
$L^1(H)$.
Its restriction to the subspace of $L$-invariant vectors, $E_x$ will be 
$L^1(H)^{\#}$-invariant and one-dimensional (cf \cite{helg2}, Ch. IV).
If the state $\varphi_x$ corresponds to the spherical function 
$\phi_x$, we
denote by $(\pi_x,H_x)$ both the representations of $H$ and of $L^1(H)$ 
that it induces.
Corresponding to this cyclic representation of $L^1(H)$ with cyclic 
unit vector $\phi_x$, we
have that the state $f \mapsto \langle \pi_x(f)\phi_x,\phi_x \rangle_x$ is
\begin{eqnarray*}
  \langle \pi_x(f)\phi_x,\phi_x \rangle_x&=&\int_H f(h)\langle 
\pi_x(h)\phi_x,\phi_x \rangle_xdh\\
&=&\int_H f(h)\langle L_h\phi_x,\phi_x \rangle_xdh\\
&=&\int_H f(h)\phi_x(h^{-1})dh\\
&=&\int_Hf(h)\overline{\phi_x(h)}dh.
\end{eqnarray*}
Therefore this representation of $L^1(H)$ is unitarily equivalent to 
the one given by
the Gelfand-Naimark-Segal correspondence, i.e., we can regard the 
representation as coming from
a representation of the group $H$.

The operator $T$ clearly intertwines the group representations $\pi$ and $\int_X\pi_xd\mu$.
The only thing that remains is to prove that $T$ is surjective. 

Suppose that $c=\{c_x\}$ is 
orthogonal to
$T(\pi(L^1(H))$, i.e., 
$$\int_X \langle \pi_x(f)\xi_x, c_x \rangle_xd\mu=0.$$
We observe that the restriction of $T$ to the space $\mathscr{H}^L$ of $L$-invariant vectors 
intertwines the representations of $\pi(L^1(H)^{\#})$ on $\mathscr{H}^L$ and $\int_x E_x d\mu$. The mapping
$$ \pi(f) \mapsto (x \mapsto \varphi_x(f))$$ is the Gelfand transform that realises the commutative $C^*$-algebra
generated by $\pi(L^1(H)^{\#})$ and the identity operator as the algebra, $C(X)$, of continuous functions on $X$.
Continuous functions of the form $\Psi(x)=\varphi_x(f^{\Psi})$, where $f^{\Psi} \in L^1(H)^{\#}$ are dense in $C(X)$. 
For such $\Psi$ we have 
\begin{eqnarray*}
\int_X \langle \pi_x(f)\xi_x, c_x \rangle_x \Psi(x)d\mu
&=&\int_X \langle \pi_x(f*f^{\Psi})\xi_x, c_x \rangle_x d\mu\\
&=&0
\end{eqnarray*}
From this we can conclude that (using once more the separability of $L^1(H)^{\#}$) 
for all $x$ outside a set of $\mu$-measure zero, the equality
$$\langle\pi_x(f)\xi_x, c_x \rangle_x=0$$ 
holds for all $f \in L^1(H)^{\#}$. Since the vectors $\xi_x$ are $L^1(H)^{\#}$-cyclic, we can conclude that $c=0$ 
and this finishes the proof.
\end{proof}

\begin{remark}
The measure $\mu$ in the above theorem is called the \emph{Plancherel 
measure} for the representation $\pi$.
\end{remark}

\subsection{Extension and expansion of the spherical functions}
Consider the mapping
$R:\mathscr{H}_{\nu} \rightarrow C^{\infty}(\mathscr{X})$ defined by
\begin{eqnarray*}
(Rf)(x)=h(x,x)^{\nu/2}f(x), x \in \mathscr{X}
\end{eqnarray*}
(see \cite{zhtams}). When $\nu>n-1$, $R$ is in fact an
$H$-intertwining operator onto a dense subspace of $L^2(\mathscr{X},d\iota)$
(where $d\iota$ is the $H$-invariant measure on $\mathscr{X}$) and the
principal series representation gives the desired decomposition of $\pi_{\nu}$ into irreducible spherical
representations.
This is a heuristic motivation for studying the functions
$R^{-1}\varphi_{\lambda}$, where $\varphi_{\lambda}$ is a
spherical function on $\mathscr{X}$.\\

\begin{thm}
Let $\nu>(n-2)/2$. The function $R^{-1}\varphi_{\lambda}(z)$ is holomorphic on $\mathscr{D}$ 
and has the power series expansion
\begin{eqnarray*}
R^{-1}\varphi_{\lambda}(z)=\sum_k 
p_k(\lambda)e_k(z), \label{expansion}
\end{eqnarray*}
where $e_k(z)$ is the normalisation of the function $z \mapsto (zz^t)^k$ in the $\mathscr{H}_{\nu}$-norm, and 
the coefficients $p_k(\lambda)$ are polynomials of degree $2k$ of 
$\lambda$ and satisfy the orthogonality relation
a) If $\nu \geq \frac{n-1}{2}$, then

\begin{eqnarray*}
&&\frac{1}{2\pi}\int_0^{\infty}\left|\frac{\Gamma(\frac{1}{2}+i\lambda)\Gamma(\frac{n-1}{2}+i\lambda)
\Gamma(\nu-\frac{n-1}{2}+i\lambda)}{\Gamma(2i\lambda)}\right|^2p_{\nu,k}(\lambda)\overline{p_{\nu,l}(\lambda)}d\lambda\\
&=&\Gamma \left(\frac{n}{2}\right)\Gamma 
\left(\nu-\frac{n-2}{2}\right)\Gamma(\nu)\delta_{kl}.
\end{eqnarray*}
b) If $\nu < \frac{n-1}{2}$, then
\begin{eqnarray*}
&&\frac{1}{2\pi}\int_0^{\infty}\left|\frac{\Gamma(\frac{1}{2}+i\lambda)\Gamma(\frac{n-1}{2}+i\lambda)
\Gamma(\nu-\frac{n-1}{2}+i\lambda)}{\Gamma(2i\lambda)}\right|^2p_{\nu,k}(\lambda)\overline{p_{\nu,l}(\lambda)}d\lambda\\
&+&\frac{\Gamma(\nu)\Gamma(\nu-\frac{n-2}{2})\Gamma(n-1-\nu)\Gamma(\frac{n}{2}-\nu)}{\Gamma(n-1-2\nu)}\\
&&\times p_{\nu,k}\left(i\left(\nu-\frac{n-1}{2}\right)
\right)\overline{p_{\nu,l}\left(i\left(\nu-\frac{n-1}{2}\right)\right)}\\
&=&\Gamma \left(\frac{n}{2}\right)\Gamma 
\left(\nu-\frac{n-2}{2}\right)\Gamma(\nu)\delta_{kl}.
\end{eqnarray*}
\end{thm}

\begin{proof}
Recall the root space decomposition for $\mathfrak{h}$. Let $\langle \,,\, 
\rangle$ denote the
inner product on $\mathfrak{a}_{\mathbb{C}}$ that is dual to the 
restriction of the Killing form to $\mathfrak{a}$.
Let $\alpha_0$ denote $\alpha/\langle \alpha, \alpha \rangle$.\\
In this setting the spherical function $\varphi_{\lambda}$ is determined 
by the formula (cf \cite{helg2}, Ch. IV, exercise 8)

\begin{eqnarray}
\varphi_{\lambda}(\exp(t\xi_e)0)=\,_2 
F_1(a',b',c';-\sinh(\alpha(t\xi_e))^2), \label{sph}
\end{eqnarray}
where
\begin{eqnarray*}
a'&=&\frac{1}{2}\left(\frac{1}{2}m_{\alpha}+m_{2\alpha}+\langle 
i\lambda,\alpha_0 \rangle 
\right)=\frac{1}{2}\left(\frac{n-1}{2}+i\lambda\right),\\
b'&=&\frac{1}{2}\left(\frac{1}{2}m_{\alpha}+m_{2\alpha}-\langle 
i\lambda,\alpha_0 
\rangle\right)=\frac{1}{2}\left(\frac{n-1}{2}-i\lambda\right),\\
c'&=&\frac{1}{2}\left(\frac{1}{2}m_{\alpha}+m_{2\alpha}+1 
\right)=\frac{1}{2}\left(\frac{n+1}{2}\right).
\end{eqnarray*}
Letting $x=\exp(t\xi_e)0=\tanh t$, \eqref{sph}  takes the form
\begin{eqnarray}
\varphi_{\lambda}(x)=\,_2 F_1(a',b',c';\frac{xx^t}{1-xx^t})
\end{eqnarray}
By Euler's formula (cf \cite{gasp}) we have
\begin{eqnarray*}
\varphi_{\lambda}(x)=\,_2 
F_1(a',b',c';\frac{xx^t}{1-xx^t})=(1-xx^t)^{a'}\,_2 F_1(a',c'-b',c';xx^t)
\end{eqnarray*}
For the function $R^{-1}\varphi_{\lambda}$ we thus get the expression
\begin{eqnarray}
R^{-1}\varphi_{\lambda}(z)=(1-zz^t)^{-\nu+a'}\,_2 F_1(a',c'-b',c;zz^t) 
\label{sp2}
\end{eqnarray}
Expanding \eqref{sp2} into a power series yields
\begin{eqnarray}
R^{-1}\varphi_{\lambda}(z)=\sum_{m=0}^{\infty}\sum_{l=0}^m
\frac{(\nu-a')_{m-l}(a')_l(c'-b')_l}{(m-l)!l!(c')_l}(zz^t)^m,\label{serie} 
\end{eqnarray}
noticing that $|zz^t|<1$ for $z \in \mathscr{D}$.
Next, we use the following simple identities:
\begin{eqnarray*}
(\nu-a')_{m-l}&=&\frac{(\nu-a')_m}{(\nu-a'+(m-l))_l}=\frac{(\nu-a')_m}{(-1)^l(-(\nu-a'+m-1))_l}\\
(m-l)!&=&\frac{m!}{(m-l+1)_l}.
\end{eqnarray*}
Substitution of these in \eqref{serie} yields
\begin{eqnarray}
&&R^{-1}\varphi_{\lambda}(z)\label{hyper}\\
&&=\sum_{m=0}^{\infty}\frac{(\nu-a')_m}{m!}\sum_{l=0}^m\frac{(a')_l(c'-b')_l(-m)_l}
{(c')_l(-(\nu-a'+m-1))_l}(zz^t)^m. \nonumber
\end{eqnarray}
The inner sum in \eqref{hyper} can be recognised as a hypergeometric 
function, i.e., we have
\begin{eqnarray*}
\sum_{l=0}^m\frac{(a')_l(c'-b')_l(-m)_l}{(c')_l(-(\nu-a'+m-1))_l}=
\,_3F_2(a',c'-b',-m;c',-(\nu-a'+m-1);1).
\end{eqnarray*}
Now we use Thomae's transformation rule (cf \cite{gasp}) for the function 
$_3F_2$:
\begin{eqnarray*}
\lefteqn{{}_3F_2(a',c'-b',-m;c',-(\nu-a'+m-1);1)}\\
&=&\frac{(-(\nu-a'+m-1)-(c'-b'))_m}{(-(\nu-a'+m-1))_m}\\
&\times&\,_3F_2(c'-a',c'-b',-m;1+(c'-b')+(\nu-a'+m-1)-m;1)
\end{eqnarray*}
We finally obtain the following expression:
\begin{eqnarray*}
R^{-1}\varphi_{\lambda}(z)=\sum_{k=0}^{\infty}c_{n,\nu,k}(\lambda)(zz^t)^k,
\end{eqnarray*}
where
\begin{eqnarray*}
c_{n,\nu,k}(\lambda)=\frac{(\nu-\frac{n-2}{2})_k}{k!}{ 
}_3F_2(-k,\frac{1+i\lambda}{2},\frac{1-i\lambda}{2};\frac{n}{2},
\nu-\frac{n-2}{2};1)
\end{eqnarray*}
Recall the continuous dual Hahn polynomials (cf \cite{wilson})
\begin{eqnarray}
S_k(x^2;a,b,c)=(a+b)_k(a+c)_k\\
\times{}_3F_2(-k,a+ix,a-ix;a+b,a+c;1)\nonumber
\end{eqnarray}
We can thus write
\begin{eqnarray*}
R^{-1}\varphi_{\lambda}(z)
&=&\sum_{k=0}^{\infty}\frac{(\nu-\frac{n-2}{2})_k}{(\frac{n}{2})_k(\nu-\frac{n-2}{2})_k 
k!}
S_k\left((\frac{\lambda}{2})^2;\frac{1}{2},\frac{n-1}{2},\nu-\frac{n-2}{2}\right)(zz^t)^k\\
&=&\sum_{k=0}^{\infty}p_{\nu,k}(\lambda)\frac{(zz^t)^k}{\|(zz^t)^k\|_{\nu}}.
\end{eqnarray*}
For the orthogonality relation in the claim, we refer to \cite{wilson}.
\end{proof}

\subsection{Principal and complementary series representations}
In this section we let $\mu$ (=$\mu_{\nu}$) be the finite measure on
the real line that orthogonalises the coefficients $p_k(\lambda)$ in
\eqref{expansion}. Let $\Lambda_{\nu}$ be its support. As we saw
above, $\mu$ can, depending on the value of $\nu$, either be
absolutely continuous with respect to Lebesgue measure or have a
point mass at $\lambda=i(\nu-(n-1)/2)$, i.e., we either have
$$\Lambda_{\nu}=(0,\infty) \bigcup \{i(\nu-(n-1)/2)\} , \,\nu \in 
((n-2)/2, (n-1)/2)$$ or
$$\Lambda_{\nu}=(0,\infty), \,\nu \geq(n-1)/2.$$
We will now construct explicit realisations for the spherical representations $\pi_{\lambda}$ corresponding
to the points $\lambda \in \Lambda_{\nu}$ on Hilbert spaces $H_\lambda$.
For $\lambda$ in the continuous part in  $\Lambda$, the underlying space $H_{\lambda}$ will be $L^2(S^{n-1})$ and
for the discrete point $i(\nu-(n-1)/2)$, $H_{\lambda}$ will be a Sobolev space.\\ 
We will hereafter
suppress the index $\nu$ and simply denote the support of $\mu$ by 
$\Lambda$.

\begin{lemma} \label{surface}
If $g \in H$, then $g$ transforms the surface measure, $\sigma$, on 
$S^{n-1}$ as
\begin{eqnarray*}
d\sigma(g \zeta)=J_{g}(\zeta)^{\frac{n-1}{n}}d\sigma(\eta).
\end{eqnarray*}
\end{lemma}

\begin{proof}
Clearly it suffices to prove the statement for automorphisms of the form
$$g=\exp{\xi_v}, v \in \mathbb{R}^n.$$
Moreover we can assume that $\zeta=e_1$, since any $\zeta \in S^{n-1}$ can 
be written as $le_1$, where
$l \in L$, and
\begin{align*}
\exp{\xi_v} (le_1)=(\exp{\xi_v} l)(e_1)=(ll^{-1}\exp{\xi_v} 
l)(e_1)=(l \sigma_{l^{-1}}(\exp \; \xi_v))
(e_1)\\=l \exp \; (Ad(l^{-1})\xi_v)(e_1)=l\exp{\xi_{\;l^{-1}v}}(e_1).
\end{align*}
Consider now the tangent space of $\mathbb{R}^n$ at $e_1$. We have an 
orthogonal decomposition
\begin{eqnarray*}
T_{e_1}(\mathbb{R}^n)=T_{e_1}(S^{n-1}) \oplus \mathbb{R}e_1.
\end{eqnarray*}
At $ge_1$ we have the corresponding
decomposition
\begin{eqnarray*}
T_{ge_1}(\mathbb{R}^n)=T_{ge_1}(S^{n-1}) \oplus \mathbb{R}ge_1.
\end{eqnarray*}
Since $H$ preserves $S^{n-1}$, 
$$dg(e_1)\,T_{e_1}(S^{n-1})=T_{ge_1}(S^{n-1}),$$ and by completing
$e_1$ and $ge_1$ to orthonormal bases for their respective tangent spaces, 
$dg(e_1)$
corresponds to a matrix of the form
\begin{align*}
\left( \begin{array}{cccc}
c & & 0 &\\
| &* &* &* \\
v &* & *&* \\
| &* &* &*
\end{array} \right)
\end{align*}
Hence 
\begin{eqnarray}
J_g(e_1)=cJ_{g|_{S^{n-1}}}(e_1),
\end{eqnarray} 
where 
\begin{eqnarray}
c=(dg(e_1)e_1,ge_1).
\end{eqnarray}
We next determine this constant $c$.

We have $$c=(dg(e_1)e_1,ge_1)=\lim_{r \rightarrow 1}(dg(re_1)re_1,gre_1).$$
For fixed $r<1$ we have
\begin{eqnarray*}
\exp{\xi_v}(re_1)&=&u+B(u,u)^{1/2}B(re_1,-u)^{-1}(re_1+Q(re_1)u)\\
&=&u+dg(re_1)(re_1+Q(re_1)u),
\end{eqnarray*}
and
\begin{eqnarray}
J_g(re_1)=\left(\frac{h(re_1,-u)}{h(u,u)^{1/2}}\right)^{-n},\label{Jacobi}
\end{eqnarray}
where $u=\tanh v$.
Since $Q(re_1)u=2(u,re_1)re_1-u$, we get
\begin{eqnarray}
\lefteqn{(dg(re_1)re_1,g(re_1))} \label{transf}\\
&=&(1+2(u,re_1))|dg(re_1)re_1|^2+(dg(re_1)re_1,u-dg(re_1)u) \nonumber
\end{eqnarray}
For any $z \in \mathscr{D} \bigcap \mathbb{R}^n$ and $v,w \in 
\mathbb{R}^n$, the identity
\begin{eqnarray}
(dg(z)v,w)=\frac{h(gz,gz)}{h(z,z)}(v,dg(z)^{-1}w) \label{metrik}
\end{eqnarray}
can be established using the transformation properties of the function $h$ and the operator $B$.
Applying \eqref{metrik} in the cases $z=re_1$, $v=re_1$, and 
$w=dg(re_1)re_1$ and \\
$w=u-dg(re_1)u$, repectively, yields
\begin{eqnarray}
(dg(re_1)re_1,dg(re_1)re_1)=\frac{h(g(re_1),g(re_1))}{h(re_1,re_1)}r^2 
\label{metrik1}
\end{eqnarray}
and
\begin{eqnarray}
&&(dg(re_1)re_1,u-dg(re_1)u)\label{metrik2}\\
&&=\frac{h(g(re_1),g(re_1))}{h(re_1,re_1)}(re_1,dg(re_1)^{-1}u-u). 
\nonumber
\end{eqnarray}
The expressions above and an elementary computation shows that \eqref{transf} can be written as 
\begin{eqnarray}
&&(dg(re_1)re_1,g(re_1)) \label{final}\\
&&=\frac{h(g(re_1),g(re_1))}{h(re_1,re_1)}\;r^2\; 
\frac{1+2(u,re_1)+|u|^2}{1-|u|^2} \nonumber
\end{eqnarray}
By the transformation rule for the Bergman kernel
\begin{eqnarray*}
h(g(re_1),g(re_1))= |J_g(re_1)|^{2/n}h(re_1,re_1).
\end{eqnarray*}
So,
\begin{eqnarray*}
c&=&\lim_{r \rightarrow 1}\;|J_g(re_1)|^{2/n}r^2 
\;\frac{1+2(u,re_1)+|u|^2}{1-|u|^2}\\
&=&|J_g(e_1)|^{2/n}\frac{h(e_1,-u)}{h(u,u)^{1/2}}.\\
\end{eqnarray*}
Comparing with the expression \eqref{Jacobi}, we have determined the constant $$c=J_g(e_1)^{1/n},$$
and this finishes the proof. 
\end{proof}
For $\lambda$ in the continuous part of $\Lambda$, the corresponding representation is a principal series representation described by 
the following proposition. (We will hereafter follow 
Helgason and in this context denote $S^{n-1}$ by $B$. The measure $\sigma$ 
will be denoted by $db$.)

\begin{prop}
For any real number $\lambda$, the map $h \mapsto \tau_{\lambda}(h)$, where
\begin{eqnarray*}
\tau_{\lambda}(h)f(b)=J_{h^{-1}}(b)^{\frac{i\lambda+\rho}{n}}f(h^{-1}b)
\end{eqnarray*}
defines a unitary representation of $H$ on $L^2(B).$
\end{prop}

\begin{proof}
We have
\begin{eqnarray*}
\int_B 
|J_{h^{-1}}(b)^{\frac{i\lambda+\rho}{n}}|^2|f(h^{-1}b)|^2db&=&\int_B 
J_{h^{-1}}(hb)^{\frac{2\rho}{n}}|f(b)|^2d(hb)\\
&=&\int_B J_{h}(b)^{-\frac{2\rho}{n}}|f(b)|^2J_h(b)^{\frac{n-1}{n}}db\\
&=&\int_B |f(b)|^2db,
\end{eqnarray*}
where the last equality follows by lemma~\ref{surface}.
\end{proof}
It is well known that the representations $\tau_{\lambda}$ above are 
unitarily equivalent to
the canonical spherical representations associated with the corresponding 
functionals $\lambda$ on
$\mathfrak{a}_{\mathbb{C}}$ (cf \cite{knapp1}, ch. 7).

In order to realise the representation $\tau_{\lambda}$ for $\lambda=i(\nu-(n-1)/2)$, we consider the following
Hilbert spaces.
\begin{defin}
For $\frac{n-2}{2n} \leq \alpha < \frac{n-1}{2n}$, let 
$\mathscr{C}_{\alpha}$ be the Hilbert space completion of the
$C^{\infty}$-functions on $S^{n-1}$ with respect to the norm
\begin{eqnarray*}
\|f\|_{\mathscr{C}_{\alpha}}=\int_{S^{n-1}}\int_{S^{n-1}}f(\zeta)\overline{f(\eta)}K(\zeta,\eta)^{\alpha}
d\sigma(\zeta)d\sigma(\eta)
\end{eqnarray*}
\end{defin}
Using the action of $H$ on $S^{n-1}$, we can define a unitary 
representation of $H$ on $\mathscr{C}_{\alpha}$ of the form
\begin{eqnarray*}
\sigma_{\alpha}:\,f \mapsto J_{h^{-1}}(\cdot)^{\beta}f(h^{-1} \cdot), h 
\in H,
\end{eqnarray*}
where $\beta=-\alpha+(n-1)/n$. The unitarity follows from
\begin{eqnarray*}
\lefteqn{\int_{S^{n-1}}\int_{S^{n-1}}J_{h^{-1}}(\zeta)^{\beta}f(h^{-1}\zeta)\overline
{J_{h^{-1}}(\eta)^{\beta}
f(h^{-1}\eta)}K(\zeta,\eta)^{\alpha}
d\sigma(\zeta)d\sigma(\eta)}\\
&=&\int_{S^{n-1}}\int_{S^{n-1}}J_{h}(\zeta)^{-\beta}f(\zeta)\overline{J_{h}(\eta)^{-\beta}
f(\eta)}K(h\zeta,h\eta)^{\alpha}J_{h}(\zeta)^{\frac{n-1}{n}}J_{h}(\eta)^{\frac{n-1}{n}}
d\sigma(\zeta)d\sigma(\eta)\\
&=&\int_{S^{n-1}}\int_{S^{n-1}}J_{h}(\zeta)^{-\beta -\alpha + 
\frac{n-1}{n}}
J_{h}(\eta)^{-\beta -\alpha + \frac{n-1}{n}}
f(\zeta)\overline{f(\eta)}K(\zeta,\eta)^{\alpha}d\sigma(\zeta)d\sigma(\eta).
\end{eqnarray*}
In fact, this representation is irreducible (cf \cite{cow}). 
We denote this representation by $\sigma_{\alpha}$. One can prove that 
for $\alpha=\nu/n$ and $\lambda=\nu-(n-1)/2$, $\sigma_{\alpha}$ and $\tau_{\lambda}$ are unitarily
equivalent.

Recall the expression in Cor. \ref{cayleycor} for the spherical functions.
In this setting we write it as
\begin{eqnarray*}
\varphi_{\lambda}(x)=\int_Be_{\lambda,b}(x)db,
\end{eqnarray*}
where
$$e_{\lambda,b}(x)=\left(\frac{h(x,x)^{1/2}}{h(x,b)}\right)^{i\lambda+\rho}$$ 
by
Cor. \ref{cayleycor}. For fixed $z \in \mathscr{D}$ and $\lambda \in
\Lambda$, $R^{-1}e_{\lambda,b}(z)$ is a function in $L^2(B)$.
Moreover, $\pi_{\nu}(H)$ makes sense as a group of mappings on
$\mathcal{O}(\mathscr{D})$, the set of holomorphic functions on
$\mathscr{D}$. We have a relationship between these representations.
\begin{lemma}
For every $g \in H$ and $\lambda \in \Lambda$, 
\begin{eqnarray}
\pi_{\nu}(g)\tau_{\lambda}(g)R^{-1}e_{\lambda,b}(z)=R^{-1}e_{\lambda,b}(z).
\end{eqnarray}
Correspondingly, for $X \in \mathfrak{h}$, we have the relation
\begin{eqnarray}
\pi_{\nu}(X)R^{-1}e_{\lambda,b}(z)=-\tau_{\lambda}(X)R^{-1}e_{\lambda,b}(z).
\end{eqnarray}
\end{lemma}
The proof is straightforward by applying the transformation rules for the 
function $h(z,w)$.

\subsection{The Fourier-Helgason transform}
The purpose of this section is to construct an $H$-intertwining unitary operator between the 
Hilbert spaces $\mathscr{H}_{\nu}$ and $\int_{\Lambda}H_{\lambda}d\mu$.

Any holomorphic function, $f$, on $\mathscr{D}$ has a power series 
expansion
\begin{eqnarray}
f(z)=\sum_{\alpha}f_{\alpha}z^{\alpha},
\end{eqnarray}
where $f_{\alpha}=\frac{\partial^{\alpha}f}{\alpha!\partial
z^{\alpha}}(0).$ We can collect the powers of equal homogeneous
degree together and write
\begin{eqnarray}
f(z)=\sum_k f_k(z),
\end{eqnarray}
where $f_k$ is of homogeneous degree $k$.
We now consider the mapping $$( \cdot, \cdot)_{\nu}: \mathcal{P} \times 
\mathcal{O}(\mathscr{D})
\rightarrow \mathbb{C}$$ defined as
\begin{eqnarray}
(f,g)_{\nu}=\sum_k \langle f,g_k\rangle_{\nu}.
\end{eqnarray}
Observe that the definition makes sense since every polynomial is 
orthogonal to all but finitely many $g_k$.
\begin{defin}
If f is a polynomial in $\mathscr{H}_{\nu}$, its generalised 
Fourier-Helgason transform is the function $\tilde{f}$
on $\Lambda \times B$ defined by
\begin{eqnarray}
\tilde{f}(\lambda,b)=(f,R^{-1}e_{\lambda,b})_{\nu}
\end{eqnarray}
\end{defin}

\begin{prop} \label{RI}
(i)\,If the polynomial $f$ is in $\mathscr{H}_{\nu}^L$, then $\tilde{f}$ 
is $L$-invariant and
\begin{eqnarray*}
\|f\|_{\nu}^2=\int_{\Lambda}\|\tilde{f}\|_{\lambda}^2d\mu,
\end{eqnarray*}
where $\|\cdot\|_{\lambda}$ is the norm on $H_{\lambda}$,
and the Fourier-Helgason transform extends to an isometry from 
$\mathscr{H}_{\nu}^L$ onto
$L^2(\Lambda, d\mu)$.\\
(ii) The inversion formula for $L$-invariant polynomials
\begin{eqnarray}
f(z)=\int_{\Lambda}\tilde{f}(\lambda)R^{-1}\varphi_{\lambda}(z)d\mu(\lambda) 
\label{radialinv}
\end{eqnarray}
holds.
Moreover, the above formula holds for arbitrary $L$-invariant functions, 
when restricted to the
submanifold $\mathscr{X}$.
\end{prop}

\begin{proof}
Writing
\begin{eqnarray*}
R^{-1}e_{\lambda,b}=\sum_{\alpha} 
c_{\alpha}(\lambda,b)z^{\alpha}=\sum_ke_{\lambda,b,k}
\end{eqnarray*}
and
\begin{eqnarray*}
R^{-1}\varphi_{\lambda}(z)=\sum_{\alpha}c_{\alpha}(\lambda)z^{\alpha}=\sum_kp_k(\lambda)e_k(z),
\end{eqnarray*}
we see that the coefficients and polynomials of homogeneous degree $k$ are 
related by
\begin{eqnarray}
c_{\alpha}(\lambda)=\int_Bc_{\alpha}(\lambda,b)db
\end{eqnarray}
and
\begin{eqnarray}
p_k(\lambda)e_k(z)=\int_Be_{\lambda,b,k}(z)db
\end{eqnarray}
respectively.
Therefore we have
\begin{eqnarray*}
\tilde{f}(\lambda,b)&=&\sum_k \langle f,e_{\lambda,b,k}\rangle_{\nu}\\
&=&\sum_k \langle \int_L\pi_{\nu}(l)fdl,e_{\lambda,b,k}\rangle_{\nu}\\
&=&\sum_k \langle f, 
\int_L\pi_{\nu}(l^{-1})e_{\lambda,b,k}dl\rangle_{\nu}\\
&=&\sum_k \langle f, 
\int_L\pi_{\lambda}(l)e_{\lambda,b,k}dl\rangle_{\nu}\\
&=&(f,R^{-1}\varphi_{\lambda})_{\nu}.
\end{eqnarray*}
This proves the $L$-invariance.
Moreover, we have
\begin{eqnarray*}
(f,R^{-1}\varphi_{\lambda})_{\nu}=\sum_k \overline{p_k(\lambda)}\langle 
f,e_k \rangle_{\nu}.
\end{eqnarray*}
Hence
\begin{eqnarray*}
\int_{\Lambda}\|\tilde{f}\|_{\lambda}^2d\mu=\sum_k |\langle 
f,e_k \rangle_{\nu}|^2=\|f\|_{\nu}^2.
\end{eqnarray*}
This proves the first part of the claim.

To prove the inversion formula, we now let $f$ be an $L$-invariant
polynomial and $x$ be a point in
$\mathscr{D} \bigcap \mathbb{R}^n$. Since we have an estimate of the form
\begin{eqnarray}
|R^{-1}\varphi_{\lambda}(x)| \leq (1-|x|^2)^{-\frac{\nu}{2}}C(x), 
\label{uppsk}
\end{eqnarray}
where $C$ is some function of $x$,
independently of $\lambda$, the integral
\begin{eqnarray*}
\int_{\Lambda}\tilde{f}(\lambda)R^{-1}\varphi_{\lambda}(x)d\mu(\lambda)
\end{eqnarray*}
makes sense for real $x$. We then have
\begin{eqnarray*}
\int_{\Lambda}\tilde{f}(\lambda)R^{-1}\varphi_{\lambda}(x)d\mu(\lambda)&=&
\sum_k \int_{\Lambda} \langle f,e_k 
\rangle_{\nu}\overline{p_k(\lambda)}R^{-1}\varphi_{\lambda}(x)d\mu(\lambda)\\
&=&\sum_k \langle f,e_k \rangle_{\nu}\int_{\Lambda} 
\sum_j\overline{p_k(\lambda)}p_j(\lambda)e_j(x)d\mu(\lambda)\\
&=&f(x).
\end{eqnarray*}
Now let $f \in \mathscr{H}_{\nu}^L$ be arbitrary. We choose a sequence of 
polynomials $f_n \in  \mathscr{H}_{\nu}^L$ such that
\begin{eqnarray*}
f=\lim f_n.
\end{eqnarray*}
Since the evaluation functionals are continuous, we have
\begin{eqnarray*}
f(x)=\lim_{n \rightarrow \infty} f_n(x)=\lim_{n \rightarrow 
\infty}\int_{\Lambda}\tilde{f_n}(\lambda)R^{-1}\varphi_{\lambda}(x)d\mu(\lambda)
\end{eqnarray*}
for every real point $x$. By Jensen's inequality and \eqref{uppsk}
\begin{eqnarray*}
&&\left|\int_{\Lambda}(\tilde{f}(\lambda)-\tilde{f_n}(\lambda))R^{-1}\varphi_{\lambda}(x)d\mu(\lambda)\right|^2 \\
&&\leq
\mu(\Lambda)\int_{\Lambda}|\tilde{f}(\lambda)-\tilde{f_n}(\lambda)|^2C(x)(1-|x|^2)^{-\nu}d\mu(\lambda).
\end{eqnarray*}
Hence
\begin{eqnarray*}
f(x)=\int_{\Lambda}\tilde{f}(\lambda)R^{-1}\varphi_{\lambda}(x)d\mu(\lambda).
\end{eqnarray*}
Thus the inversion formula holds for real points, $x$. To see that the 
formula holds for arbitrary points
when $f$ is a polynomial, we note
that both the left hand- and the right hand side of the formula define 
holomorphic functions on $\mathscr{D}$.
Since they agree on the totally real form $\mathscr{X}$, they are 
equal.
\end{proof}

\begin{thm}[The Plancherel Theorem] \label{intertw}
For $\nu>(n-2)/2$, the Fourier-Helgason transform is a unitary isomorphism from the 
$H$-modules
$\mathscr{H}_{\nu}$ onto the $H$-module 
$\int_{\Lambda}H_{\lambda}d\mu$, i.e.,
$$(\pi_{\nu}(h)f)^{\widetilde{}}(\lambda,b)=\tau_{\lambda}(h)\widetilde{f}(\lambda,b),$$
for $h \in H$, and
$$\|f\|_{\nu}^2=\int_{\Lambda}\|\tilde{f}\|_{\lambda}^2d\mu.$$
\end{thm}
\begin{proof}
We divide the proof into three steps:\\
$(i)$ We prove that the Fourier-Helgason transform intertwines the action 
of the Lie algebra
of $H$.\\
$(ii)$ We use $(i)$ to prove that the norm is preserved.\\
$(iii)$ We conclude that the group actions are intertwined from $(i)$ and 
$(ii)$.

We will see that these properties actually imply that the Fourier-Helgason 
transform is surjective.

Consider now the corresponding representations of the Lie algebra 
$\mathfrak{h}$. These will also be denoted by $\pi_{\nu}$
and $\tau_{\lambda}$ respectively. Moreover they extend naturally to 
representations of the universal
enveloping algebra, $\mathfrak{U}(\mathfrak{h})$, of $\mathfrak{h}$. \\
Let $X \in \mathfrak{h}$. If $f$ is a polynomial in $\mathscr{H}_{\nu}$, 
then differentiation of the mapping
\begin{eqnarray*}
t \mapsto J_{\exp tX}(z)^{\nu/n}f((\exp tX)z)
\end{eqnarray*}
at $t=0$ shows that
$\pi_{\nu}(X)f$ is also a polynomial, and

\begin{eqnarray*}
\widetilde{\pi_{\nu}(X)f}(\lambda,b)&=&\sum_k \langle 
\pi_{\nu}(X)f,e_{\lambda,b,k} \rangle_{\nu}\\
&=&\sum_k \langle f, -\pi_{\nu}(X)e_{\lambda,b,k} \rangle_{\nu}\\
&=&\sum_k \langle f, \tau_{\lambda}(X)e_{\lambda,b,k} \rangle_{\nu}\\
&=&(f,\tau_{\lambda}(X)R^{-1}e_{\lambda,b})_{\nu}\\
&=&\tau_{\lambda}(X)(f,R^{-1}e_{\lambda,b})_{\nu},
\end{eqnarray*}
which proves $(i)$.

To prove the second step, we recall that the adjoint representation of $L$ 
on $\mathfrak{h}$ extends
to an action on $\mathfrak{U}(\mathfrak{h})$ as homomorphisms of an 
associative algebra.
The $L$-invariant elements in $\mathfrak{U}(\mathfrak{h})$ form a 
subalgebra,
$\mathfrak{U}(\mathfrak{h})^{L}$. We let $p$ denote the projection $$X 
\mapsto \int_LAd(l)Xdl$$ of
$\mathfrak{U}(\mathfrak{h})$ onto $\mathfrak{U}(\mathfrak{h})^L$.
This action of $L$
connects the representations of $H$ and $\mathfrak{U}(\mathfrak{h})$ 
according to the following identity:
$$\pi_{\nu}(l)\pi_{\nu}(X)\pi_{\nu}(l^{-1})=\pi_{\nu}(Ad(l)X),$$ for $l \in L$ and $X \in 
\mathfrak{U}(\mathfrak{h})$.

Since the vector $1 \in \mathscr{H}_{\nu}$ is cyclic for the 
representation of $H$, it is also cyclic for
the representation of $\mathfrak{U}(\mathfrak{h})$. Hence it suffices to 
prove that the norm is
preserved for elements of the form $\pi_{\nu}(X)1$, where $X \in 
\mathfrak{U}(\mathfrak{h})$.
In the following equalities, we temporarily let $\tau$ denote the direct 
integral of the representations
$\tau_{\lambda}$, and analogously we let $\langle, \rangle$ denote the 
direct integral of the corresponding
inner products.
\begin{eqnarray*}
\langle \pi_{\nu}(X)1, \pi_{\nu}(X)1 \rangle_{\nu}&=&\langle 
\pi_{\nu}(X)^{*}\pi_{\nu}(X)1, 1 \rangle_{\nu}\\
&=&\langle -\pi_{\nu}(X^2)1, 1 \rangle_{\nu}
\end{eqnarray*}
Since the vector $1$ is $L$-invariant, the last expression equals
$\langle -\pi_{\nu}(p(X^2))1, 1 \rangle_{\nu}$, and by proposition 
\eqref{RI}, we have
$$\langle -\pi_{\nu}(p(X^2))1, 1 \rangle_{\nu}=\langle 
-\widetilde{\pi_{\nu}(p(X^2))1}, \tilde{1} \rangle.$$
By $(i)$, the expression on the right-hand side equals $\langle 
-\tau(p(X^2))\tilde{1}, \tilde{1} \rangle$,
and since $\tilde{1}$ is $L$-invariant, we have
$$\langle -\tau(p(X^2))\tilde{1}, \tilde{1} \rangle=\langle 
-\tau(X^2)\tilde{1}, \tilde{1} \rangle.$$
Thus $(ii)$ is proved.

To prove $(iii)$, we recall the following equalities (on the respective
dense spaces of analytic vectors):
\begin{eqnarray*}
\pi_{\nu}(\exp(X))=e^{\pi_{\nu}(X)}\\
\tau_{\lambda}(\exp(X))=e^{\tau_{\lambda}(X)}.
\end{eqnarray*}
From this and the facts that $H$ is connected and that the 
Fourier-Helgason transform is bounded operator, we immediately see
that $(iii)$ holds.

To see that the operator is surjective, note that
by $(ii)$ and $(iii)$
\begin{eqnarray*}
\langle \pi_{\nu}(f)1,1\rangle_{\nu}=\int_{\Lambda}\langle 
\tau_{\lambda}\tilde{1}(\lambda,\cdot),\tilde{1}(\lambda,\cdot)
\rangle_{\lambda}d\mu,
\end{eqnarray*}
for $f \in L^1(H)^{\#}$, i.e., we can write the positive functional
\begin{eqnarray*}
f \mapsto \langle \pi_{\nu}(f)1,1\rangle_{\nu}
\end{eqnarray*}
as an integral of pure states with respect to some measure. By uniqueness, 
it is the measure in
Theorem~\ref{abstr}. Since the Fourier-Helgason transform intertwines the 
group action, it is the intertwining operator
constructed in Theorem~\ref{abstr}. Thus it is surjective.
\end{proof}

\begin{thm}[The Inversion Formula]
If $f$ is a polynomial in $\mathscr{H}_{\nu}$, then
\begin{eqnarray}
f(z)=\int_{\Lambda}\int_B 
\widetilde{f}(\lambda,b)R^{-1}e_{\lambda,b}(z)dbd\mu(\lambda).
\end{eqnarray}
\end{thm}

\begin{proof}
Take $h \in H$. Define
\begin{eqnarray*}
f_1(z)=\int_L \pi_{\nu}(l)\pi_{\nu}(h)f(z)dl
\end{eqnarray*}
This is a radial function, and we have that
\begin{eqnarray}
f_1(0)=J_{h^{-1}}(0)^{\frac{\nu}{n}}f(h^{-1}0). \label{radlik}
\end{eqnarray}
Prop.~\ref{RI} gives
\begin{eqnarray}
f_1(0)=\int_{\Lambda}\tilde{f_1}(\lambda)R^{-1}\varphi_{\lambda}(z)d\mu(\lambda). 
\label{inv}
\end{eqnarray}
Moreover
\begin{eqnarray}
\widetilde{f_1}(\lambda)=(f_1,R^{-1}\varphi_{\lambda})_{\nu}
&=&(\int_L\pi_{\nu}(l)\pi_{\nu}(h)fdl,R^{-1}\varphi_{\lambda})_{\nu}\nonumber 
\\
&=&(\pi_{\nu}(h)f,R^{-1}\varphi_{\lambda})_{\nu}. \nonumber \\
\end{eqnarray}
By Thm.\ \ref{intertw} we have
\begin{eqnarray}
(\pi_{\nu}(h)f,R^{-1}\varphi_{\lambda})_{\nu} \nonumber
&=&(f,\pi_{\nu}(h^{-1})R^{-1}\varphi_{\lambda})_{\nu}\nonumber \\
&=&(f,\int_B\pi_{\nu}(h^{-1})R^{-1}e_{\lambda,b}\,db)_{\nu}\nonumber\\
&=&(f,\int_B\pi_{\nu}(h^{-1})R^{-1}e_{\lambda,b}\,db)_{\nu}\nonumber\\
&=&(f,\int_B\tau_{\lambda}(h)R^{-1}e_{\lambda,b}\,db)_{\nu}\nonumber\\
&=&(f,\int_BJ_{h^{-1}}(b)^{\frac{i\lambda+\rho}{n}} 
R^{-1}e_{\lambda,h^{-1}b}\,db)_{\nu}.\nonumber
\end{eqnarray}
The integrand above has a power series expansion where the coefficients 
are functions of $b$. If we integrate, we
obtain a holomorphic functions for which the coefficients in the power 
series expansion are obtained by integrating the aforementioned 
coefficients over $B$.
Hence we can proceed as follows.
\begin{eqnarray}
(f,\int_BJ_{h^{-1}}(b)^{\frac{i\lambda+\rho}{n}} 
R^{-1}e_{\lambda,h^{-1}b}\,db)_{\nu}
&=&\int_BJ_{h^{-1}}(b)^{\frac{-i\lambda+\rho}{n}} 
(f,R^{-1}e_{\lambda,h^{-1}b})_{\nu}db\nonumber\\
&=&\int_BJ_{h^{-1}}(b)^{\frac{-i\lambda+\rho}{n}}\widetilde{f}(\lambda,h^{-1}b)db\nonumber\\
&=&\int_BJ_{h^{-1}}(hb)^{\frac{-i\lambda+\rho}{n}}\widetilde{f}(\lambda,b)J_h(b)^{\frac{n-1}{n}}db 
\nonumber\\
&=&\int_B\widetilde{f}(\lambda,b)J_h(b)^{\frac{i\lambda+\rho}{n}}db 
\label{sista}.
\end{eqnarray}
It is easy to see that
\begin{eqnarray}
J_h(b)^{\frac{i\lambda+\rho}{n}}=J_{h^{-1}}(0)^{\frac{\nu}{n}}R^{-1}e_{\lambda,b}(h^{-1}0),
\end{eqnarray}
and so combining \eqref{radlik}, \eqref{inv} and \eqref{sista} finally 
yields
\begin{eqnarray}
f(h^{-1}0)=\int_{\Lambda}\int_B 
\widetilde{f}(\lambda,b)R^{-1}e_{\lambda,b}(h^{-1}0)db\,d\mu.
\end{eqnarray}
Thus the inversion formula holds for real points, hence for all points by 
the same argument as in the proof of
Prop.~\ref{RI}.
\end{proof}

\end{section}

\begin{section}{Realisation of the discrete part of the decomposition}
We recall the earlier defined complementary series representations.
The following theorem states that $\sigma_{\nu/n}$ is the representation 
corresponding to the singular point in the decomposition theorem.
\begin{thm}
The operator $T_{\nu}$ defined by the formula
\begin{eqnarray*}
(T_{\nu}f)(z)=\int_{S^{n-1}}f(\zeta)K_{\nu}(z,\zeta)d\sigma(\zeta)
\end{eqnarray*}
is a unitary $H$-intertwining operator from $\mathscr{C}_{\nu/n}$ onto an 
irreducible $H$-submodule of $\mathscr{H}_{\nu}$.
\end{thm}

\begin{proof}
First of all we note that $T_{\nu}$ maps functions in 
$\mathscr{C}_{\nu/n}$ to holomorphic functions on $\mathscr{D}$ and thus 
$\pi_{\nu}$ has
a meaning on the range of $T_{\nu}$. We start by showing that 
$T_{\nu}$ is formally intertwining.
We have
\begin{eqnarray*}
T_{\nu}(\sigma_{\nu/n})f(z)
&=&\int_{S^{n-1}}J_{h^{-1}}(\zeta)^{-\nu/n + 
\frac{n-1}{n}}f(h^{-1}\zeta)K_{\nu}(z,\zeta)d\sigma(\zeta)\\
&=&\int_{S^{n-1}}J_{h}(\zeta)^{\nu/n - 
\frac{n-1}{n}}f(\zeta)K_{\nu}(z,h\zeta)
J_{h}(\zeta)^{\frac{n-1}{n}}d\sigma(\zeta)\\
&=&\int_{S^{n-1}}J_{h}(\zeta)^{\nu/n}f(\zeta)K_{\nu}(h^{-1}z,\zeta)
J_{h}(h^{-1}z)^{-\frac{\nu}{n}}J_{h}(\zeta)^{-\frac{\nu}{n}}d\sigma(\zeta)\\
&=&J_{h^{-1}}(z)^{\frac{\nu}{n}}\int_{S^{n-1}}f(\zeta)
K_{\nu}(h^{-1}z,\zeta)d\sigma(\zeta),
\end{eqnarray*}
i.e., $$T_{\nu}\sigma_{\nu/n}=\pi_{\nu}T_{\nu}.$$
The next step is to prove that the constant function $1$ is mapped into 
$\mathscr{H}_{\nu}$ and that its norm is preserved.
Note that for $\alpha = \nu/n, K(z,\zeta)^{\alpha}=K_{\nu}(z, \zeta)$,
and by Prop. ~\ref{FK} we have an expansion
\begin{eqnarray*}
K_{\nu}(\zeta,e_1)=\sum_{m-2k \geq 0}c_{m,k}(\nu)K_{(m-k,k)}(\zeta,e_1),
\end{eqnarray*}
where the coefficients $c_{m,k}(\nu)$ are given explicitly.
Now, since $K_{\nu}(\zeta,e_1)$ is $SO(n-1)$-invariant and the action of 
$SO(n-1)$ is linear, each $K_{(m,k)}(\zeta,e_1)$
must also be $SO(n-1)$-invariant. Hence, $K_{(m,k)}(\zeta,e_1)$ can be 
assumed to be $\phi_{m-2k}(\zeta)
(\zeta \zeta^t)^k$,
where
$\phi_{m-2k}$ is the unique element in $E_{m-2k}$ that assumes the value 1
in $e_1$.
Therefore
\begin{eqnarray}
&&\int_{S^{n-1}}K(\zeta,\eta)^{\alpha}d\sigma(\zeta)\label{kernel21}\\
&=&\int_L K_{\nu}(\zeta, 
le_1)dl
\int_L K_{\nu}(l^{-1}\zeta,e_1)dl \nonumber\\
&=&\sum_{m-2k \geq 0}c_{m,k}(\nu)\int_L 
(l^{-1}\zeta(l^{-1}\zeta)^t)^k\phi_{m-2k}(l^{-1}\zeta)dl
\label{kernel22}\\
&=&\sum_{m-2k \geq 0}c_{m,k}(\nu)(\zeta \zeta^t)^k\int_L 
\phi_{m-2k}(l^{-1}\zeta)dl \label{kernel23}
\end{eqnarray}
Since $SO(n)$ acts irreducibly on $E_{m-2k}$ and the function $\int_L 
\phi_{m-2k}(l^{-1}z)dl$ is an
$SO(n)$-invariant element in  $E_{m-2k}$ it must be identically zero 
unless $m-2k=0$.
Since
\begin{eqnarray*}
\|1\|_{\mathscr{C}_{\frac{\nu}{n}}}^2=\int_{S^{n-1}}\int_{S^{n-1}}K_{\nu}(\zeta,\eta)d\sigma(\zeta)d\sigma(\eta),
\end{eqnarray*}
the computation above implies that
\begin{eqnarray}
\|1\|_{\mathscr{C}_{\frac{\nu}{n}}}^2 
&=&\int_{S^{n-1}}\sum_{k=0}^{\infty}c_{2k,k}(\nu)(\zeta \zeta^t)^k
d\sigma(\zeta) \nonumber\\
&=&\sum_{k=0}^{\infty}c_{2k,k}(\nu) \nonumber
=\sum_{k=0}^{\infty}\frac{(\nu)_k(\nu-\frac{n-2}{2})_k}{\|(zz^t)^k\|_{\mathscr{F}}^2} 
\nonumber\\
&=&\sum_{k=0}^{\infty}\frac{(\nu)_k(\nu-\frac{n-2}{2})_k}{k! 
(\frac{n}{2})_k}\label{1}
\end{eqnarray}
On the other hand, the equalities \eqref{kernel21}-\eqref{kernel23} also 
show that
\begin{eqnarray}
T_{\nu}1(z)
&=&\sum_{k=0}^{\infty}\frac{(\nu)_k(\nu-\frac{n-2}{2})_k}{k! 
(\frac{n}{2})_k}(zz^t)^k \nonumber\\
&=&\sum_{k=0}^{\infty}\frac{((\nu)_k(\nu-\frac{n-2}{2})_k)^{1/2}}{(k! 
(\frac{n}{2})_k)^{1/2}}\frac{(zz^t)^k}{\|(zz^t)^k\|_{\nu}}.
\label{2}
\end{eqnarray}
If we compare \eqref{1} and \eqref{2}, we see that $T_{\nu}1 \in 
\mathscr{H}_{\nu}$ and that
$\|1\|_{\mathscr{C}_{\nu/n}}=\|1\|_{\nu}.$
Recall that
\begin{eqnarray*}
\mathscr{C}_{\nu/n}=\overline{\bigoplus_m \,E_m(S^{n-1})}
\end{eqnarray*}
and that the representation of $\mathfrak{h}$ on the algebraic sum 
$\bigoplus_m \,E_m(S^{n-1})$ is irreducible. Hence
\begin{eqnarray*}
\bigoplus_m 
\,E_m(S^{n-1})=\text{Span}_{\mathbb{C}}\{\sigma_{\nu/n}(X_1)\ldots\sigma_{\nu/n}(X_k)1 
| X_i \in \mathfrak{h}, 1 \leq i \leq k\}
\end{eqnarray*}
Since $T_{\nu}$ interwines the representations of $\mathfrak{h}$, we 
have that $\pi_{\nu}$ is an irreducible representation of $\mathfrak{h}$ 
on
the space  $T_{\nu}(\bigoplus_m \,E_m(S^{n-1})) \subseteq 
\mathscr{H}_{\nu}$. By Schur's lemma (\cite{knapp2}, ch.4)
\begin{eqnarray*}
\langle T_{\nu}f,T_{\nu}g \rangle_{\nu}=c\langle f,g 
\rangle_{\mathscr{C}_{\nu/n}},
\end{eqnarray*}
for some real constant $c$.
Putting, $f$ and $g$ equal to the constant function $1$ and applying, we 
see that $c=1$.
Therefore, $T_{\nu}$ extends to a unitary operator
\begin{eqnarray*}
T_{\nu}: \mathscr{C}_{\nu/n} \rightarrow 
\overline{T_{\nu}(\bigoplus_mE_m(S^{n-1}))}
\end{eqnarray*}
and we have proved the theorem.
\end{proof}

\end{section}

\begin{section}{Realisation of the minimal representation $\pi_{(n-2)/2}$}
In this section we show that the representation $\pi_{(n-2)/2}$ of $H$ is 
irreducible by realising it as a complementary series representation.

We recall the space $\mathscr{C}_{\nu/n}$ from the previous section and 
the corresponding operator $T_{\nu}$.
\begin{thm}
$T_{(n-2)/2}$ is a unitary $H$-intertwining operator from 
$\mathscr{C}_{(n-2)/2n}$ onto $\mathscr{H}_{(n-2)/2}$.
\end{thm}

\begin{proof}
Recall that
\begin{eqnarray} \label{min}
\mathscr{C}_{(n-2)/n}=\overline{\bigoplus_m \,E_m(S^{n-1})}
\end{eqnarray}
and that the sum is a decomposition into $SO(n)$-irreducible subspaces. If 
we let $\mathcal{P}_{(n-2)/n}$ denote the set of all finite sums
in \eqref{min}, $\sigma_{(n-2)/n}$ defines a representation of 
$\mathfrak{l}$ on  $\mathcal{P}_{(n-2)/n}$. The polynomial
$(\zeta_1 + i\zeta_2)^m$ is a highest weight vector in $E_m$ for this 
representation. Moreover, the power series expansion of $K_{(n-2)/n}$
shows that $T_{(n-2)/2}$ is a polynomial in $E_m$. Since $T_{(n-2)/2}$ 
intertwines the $\mathfrak{l}$-actions, $T_{(n-2)/2}((\zeta_1 + 
i\zeta_2)^m)$ is a
highest weight vector space for $\pi_{(n-2)/2}(\mathfrak{l})$, i.e.,
\begin{eqnarray}
(T_{(n-2)/2}(\zeta_1 + i\zeta_2)^m)(z)=C_m(z+iz)^m,
\end{eqnarray}
for some constant $C_m$. We now determine $C_m$.
Choose $z=w \frac{1}{2}(1,-i,0,\ldots ,0)$, where $w$ is a complex number 
with $|w|<1$. In this case $zz^t=0, (z+iz)^m=w^m$. We now
compute $(T_{(n-2)/2}((\zeta_1+i\zeta_2)^m)(z)$.
\begin{eqnarray*}
\lefteqn{\int_{S^{n-1}}K_{(n-2)/n}(z,\zeta)(\zeta_1 + i\zeta_2)^m}\\
&=&\int_{S^{n-1}}(1-w(\zeta_1-i\zeta_2))^{-(n-2)/n}(\zeta_1+i\zeta_2)^md\sigma(\zeta)
\end{eqnarray*}
This integral only depends on the first two coordinates and can hence be 
converted to an integral over the unit disk, $U$ (cf \cite{rudinft} Prop 
1.4.4).
\begin{eqnarray*}
\lefteqn{\int_{S^{n-1}}(1-w(\zeta_1-i\zeta_2))^{-(n-2)/n}(\zeta_1+i\zeta_2)^md\sigma(\zeta)}\\
&=&\frac{\Gamma \left(\frac{n-2}{2}\right)}{\pi \Gamma 
\left(\frac{n}{2}\right)}
\int_U(1-w\overline{\zeta})^{-(n-2)/n}\zeta^m(1-|\zeta|^2)^{(n-4)/2}dm(\zeta).
\end{eqnarray*}
We have the power series expansion
\begin{eqnarray*}
(1-w\overline{\zeta})^{-(n-2)/2}=\sum_{k=0}^{\infty}\left(\frac{n-2}{2}\right)_k(z\overline{\zeta})^k
\end{eqnarray*}
Recall that $(1-w\overline{\zeta})^{-n/2}$ is the reproducing kernel for 
the weighted Bergman space $\mathscr{H}_{n/2}(U)$, defined as
\begin{eqnarray*}
\mathscr{H}_{n/2}(U)=\{f \in 
\mathcal{O}(U)\,|\,\frac{\Gamma\left(\frac{n}{2}\right)}{\pi\Gamma 
\left(\frac{n-2}{2}\right)}
\int_U|f(\zeta)|^2(1-|\zeta|^2)^{(n-4)/2}dm(\zeta)<\infty\},
\end{eqnarray*}
Polynomials of different degree are orthogonal in $\mathscr{H}_{n/2}(U)$ 
and hence we have
\begin{eqnarray*}
\lefteqn{\int_U(1-w\overline{\zeta})^{-(n-2)/n}\zeta^m(1-|\zeta|^2)^{(n-4)/2}dm(\zeta)}\\
&=&\int_U 
\sum_{k=0}^{\infty}\left(\frac{n-2}{2}\right)_k(z\overline{\zeta})^k 
\zeta^m(1-|\zeta|^2)^{(n-4)/2}dm(\zeta)\\
&=&\int_U 
\sum_{k=0}^{\infty}\left(\frac{n-2}{2}\right)\left(\frac{n}{2}\right)_m(z\overline{\zeta})^m 
\zeta^m(1-|\zeta|^2)^{(n-4)/2}
dm(\zeta)\\
&=&\pi w^m,
\end{eqnarray*}
where the last equality follows from the reproducing property in 
$\mathscr{H}_{n/2}(U)$.
Summing up, we have
\begin{eqnarray}
(T_{(n-2)/2}(\zeta_1+i\zeta_2)^m)(z)=\frac{n-2}{2\pi^2}(z_1+iz_2)^m
\end{eqnarray}
From this and the intertwining of the $\mathfrak{l}$-action, it follows 
that
\begin{eqnarray}
T_{(n-2)/2}\left(\bigoplus_m E_m(S^{n-1})\right) \subseteq \bigoplus_m 
E_m
\end{eqnarray}
To compute the norm of $T_{(n-2)/2}(p)$ where $p \in E_k(S^{n-1})$, we 
first fix $r<1$ and consider the polynomial $T_{(n-2)/2}(p(rz))$.
By definition
\begin{eqnarray}
T_{(n-2)/2}(p)(rz)&=&\int_{S^{n-1}}K_{\nu}(rz,\zeta)p(\zeta)d\sigma(\zeta)\nonumber\\
&=&\int_{S^{n-1}}K_{\nu}(z,r\zeta)p(\zeta)d\sigma(\zeta).\label{riemann}
\end{eqnarray}
The norm is given by
\begin{eqnarray*}
\|T_{(n-2)/2}(p)(r\,\cdot)\|_{\nu}^2=\int_{S^{n-1}}\int_{S^{n-1}}p(\zeta)\overline{p(\eta)}K_{\nu}(r\zeta,r\eta)d\sigma(\zeta)d\sigma(\eta).
\end{eqnarray*}
Finally, we let $r \rightarrow 1$ and obtain
\begin{eqnarray*}
\|T_{(n-2)/2}(p)\|_{\nu}^2=\int_{S^{n-1}}\int_{S^{n-1}}p(\zeta)\overline{p(\eta)}K_{\nu}(\zeta,\eta)d\sigma(\zeta)d\sigma(\eta).
\end{eqnarray*}
From this and the orthogonality of the spaces $E_k$, it follows that 
$T_{(n-2)/2}$ maps $(\bigoplus_m E_m(S^{n-1})$ isometrically onto
$(\bigoplus_m E_m)$. Hence it extends to a unitary operator from 
$\mathscr{C}_{(n-2)/2n}$ onto $\mathscr{H}_{(n-2)/2}$.
\end{proof}

\end{section}

\end{document}